\documentclass[11pt,reqno]{amsart}
\usepackage{fullpage,times,graphicx,amssymb,amsmath,xcolor}
\usepackage[utf8]{inputenc}
\usepackage[square,numbers]{natbib}
\usepackage[colorlinks,citecolor=bbluegray,linkcolor=ddarkbrown,urlcolor=blue,breaklinks]{hyperref}
\usepackage{bbding}

\usepackage{subcaption}
\usepackage{algorithmic,algorithm}

\usepackage{comment}

\definecolor{ddarkbrown}{rgb}{0.5,0.2,0.05} \definecolor{bbluegray}{rgb}{0.05,0,0.5}

\newtheorem{theorem}{Theorem}[section]
\newtheorem{proposition}[theorem]{Proposition}
\newtheorem{observation}[theorem]{Observation}
\newtheorem{definition}[theorem]{Definition}
\newtheorem{lemma}[theorem]{Lemma}
\newtheorem{corollary}[theorem]{Corollary}

\newcommand{\LP}[1]{\ensuremath{\operatorname{LP}\sb{#1}}}

\newcommand{\BEAS}{\begin{eqnarray*}}
\newcommand{\EEAS}{\end{eqnarray*}}
\newcommand{\BEA}{\begin{eqnarray}}
\newcommand{\EEA}{\end{eqnarray}}
\newcommand{\BEQ}{\begin{equation}}
\newcommand{\EEQ}{\end{equation}}
\newcommand{\BIT}{\begin{itemize}}
\newcommand{\EIT}{\end{itemize}}
\newcommand{\BNUM}{\begin{enumerate}}
\newcommand{\ENUM}{\end{enumerate}}

\newcommand{\BA}{\begin{array}}
\newcommand{\EA}{\end{array}}

\newcommand{\ones}{\mathbf 1}

\newcommand{\reals}{{\mathbb R}}

\newcommand{\Co}{{\mathop {\bf Co}}}

\newcommand{\Ext}{{\mathop {\bf Ext}}}

\newcommand{\argmin}{\mathop{\rm argmin}}

\let \oldsection \section
\renewcommand{\section}{\vspace{3ex plus 1ex}\oldsection}

\makeatletter
\def\@setdate{\@date}
\makeatother

\begin{document}

\title{Restarting Frank-Wolfe: Faster Rates under H\"olderian Error Bounds}

\author{Thomas Kerdreux $^{\dagger,\ast}$}
\address{Zuse Institute Berlin \& Technische Universit\"at Berlin, Germany}
\email{thomaskerdreux@gmail.com}

\author{Alexandre d'Aspremont$^{\ddagger,\mathsection}$}
\address{CNRS \& D.I., UMR 8548,\vskip 0ex
\'Ecole Normale Sup\'erieure, Paris, France.}
\email{aspremon@ens.fr}

\author{Sebastian Pokutta$^{\dagger,\ast}$}
\address{Zuse Institute Berlin \& Technische Universit\"at Berlin, Germany}
\email{pokutta@zib.de}



\date{$^\dagger$Zuse Institute, Berlin, Germany.\\
\indent$^\ast$Technische Universit{\"a}t, Berlin, Germany.\\
\indent$^\ddagger$CNRS UMR 8548.\\
\indent$^\mathsection$D.I. \'Ecole Normale Sup\'erieure, Paris, France.}

\subjclass[2010]{}

\maketitle

\begin{abstract}
Conditional Gradient algorithms (aka Frank-Wolfe algorithms) form a classical set of methods for constrained smooth convex minimization due to their simplicity, the absence of projection steps, and competitive numerical performance.
While the vanilla Frank-Wolfe algorithm only ensures a worst-case rate of $\mathcal{O}(1/\epsilon)$, various recent results have shown that for strongly convex functions on polytopes, the method can be slightly modified to achieve linear convergence. However, this still leaves a huge gap between sublinear $\mathcal{O}(1/\epsilon)$ convergence and linear $\mathcal{O}(\log 1/\epsilon)$ convergence to reach an $\epsilon$-approximate solution.
Here, we present a new variant of Conditional Gradient algorithms, that can dynamically adapt to the function's geometric properties using restarts and smoothly interpolates between the sublinear and linear regimes.
These interpolated convergence rates are obtained when the optimization problem satisfies a new type of error bounds, which we call \textit{strong Wolfe primal bounds}. They combine geometric information on the constraint set with H\"olderian Error Bounds on the objective function.
\end{abstract}

\section{Introduction}
\label{sec:introduction}
We consider smooth constrained convex minimization, solving problems of the form 
\begin{equation}\label{eq:optim_problem}
\min_{x \in \mathcal C} f(x),
\end{equation}
where $f$ is a smooth convex function and $\mathcal C$ is a compact convex set. As soon as the geometry of $\mathcal C$ is reasonably complicated, so that projections onto the set are computationally expensive, projection-free first-order methods such as Conditional Gradient algorithms \citep{levitin1966constrained} (also known as Frank-Wolfe methods \citep{Fran56}) become an efficient alternative as they only require first-order access to the function under consideration as well as access to an efficient linear optimization oracle for the feasible region $\mathcal C \subseteq \reals^n$ which, given a linear objective $c \in \reals^n$, outputs $\arg\min_{x \in \mathcal C} c^T x$.

In order to reach an $\epsilon$-approximate solution $\hat x$, so that $f(\hat x) - f(x^*) < \epsilon$, where $x^*$ is an optimal solution, the standard Frank-Wolfe algorithm requires a number of iterations of order $O(1/\epsilon)$, that cannot be improved in general \citep{canon1968tight,jaggi2013revisiting}.
A series of recent works (see e.g., \citep{garber2013linearly,FW-converge2015}; see also \citep{lan2014conditional} for conditional gradient sliding) showed that when $f$ is strongly convex the convergence rate of the standard case can be improved to $O(\log 1/\epsilon)$ and various extensions further improved upon these results for special cases (see e.g., \citep{lacoste2013block,garber2015faster,Freund2016,garber2016linear,braun2017lazifying,lan2017conditional,bashiri2017decomposition,kerdreux2018frank,braun2019blended,YiAdapativeFW,carderera2019locally}), applying Frank-Wolfe methods to machine learning problems (e.g., \cite{joulin2014efficient,shah2015multi,osokin2016minding,locatello2017unified,freund2017extended,locatello2017greedy,miech2017learning}). Nonetheless, these results left a wide gap between the linear $O( \log 1/\epsilon)$ rate and the sub-linear $O(1/\epsilon)$ rate.
Note that \cite{kerdreux2021projection} also obtain such interpolated rates when analysing the vanilla Frank-Wolfe algorithm on (locally) uniformly convex constraint sets, extending the known accelerated regimes of Frank-Wolfe when the constraint set is strongly convex.

Here, we present a new variant of the Conditional Gradient method using the scaling argument of the parameter-free Lazy Frank-Wolfe variant in \citep{braun2017lazifying,braun2019blended}, together with a restart scheme similar to that used for gradient methods in e.g., \citep{Nemi85,Gise14,ODon15,Ferc16,roulet2020sharpness}. This yields an algorithm that dynamically adapts to the local properties of the function and the feasible region around the optimum. The convergence proof relies on two key conditions. One is a scaling inequality (Definition \ref{def:Scaling}) used to characterize the regularity of $\mathcal{C}$ in many Frank-Wolfe complexity bounds which holds on e.g., polytopes and strongly convex sets. The other is a local growth condition which is known to hold generically for sub-analytic functions by the {\L}ojasiewicz factorization lemma (see e.g., \citep{bolte2007lojasiewicz}) and controls for example the impact of restart schemes as in \citep{roulet2020sharpness}.

Earlier work showed that a sharpness condition derived from the {\L}ojasiewicz lemma could be used to improve convergence rates of gradient methods (see e.g., \citep{Nemi85,bolte2007lojasiewicz,Atto10,Baus16,karimi2016linear} for an overview). However, to achieve improved rates, these methods required exact knowledge of the constants appearing in the condition, which are in practice typically not observed. In contrast to this, as in \citep{roulet2020sharpness,chen2018sadagrad}, we show using robust restart schemes that our algorithm does not require knowledge of these constants, thus making it essentially parameter-free.

\subsection*{Contributions}
This paper is a journal version of \citep{kerdreux2019restarting}.
Our contributions can be summarized as follows.

\begin{enumerate}
  \item \emph{Strong Wolfe primal bound.} Under generic assumptions, we derive strong Wolfe primal gap bounds generalizing those obtained from strong convexity of $f$. These bounds are obtained by combining a {\L}ojasiewicz growth condition on $f$ with a scaling inequality on $\mathcal{C}$, and continuously interpolate between the convex and strongly convex cases. 
  Section \ref{sec:prelims} and \ref{sec:LGC} provide a more in-depth approach than in \citep{kerdreux2019restarting}.
  
  \item \emph{Fractional Frank-Wolfe Algorithms.} We then define a new Conditional Gradient algorithm that dynamically adapts to the parameters of these strong Wolfe primal bounds using a restart scheme. The resulting algorithm achieves either sub-linear (i.e., $O(1/\epsilon^q)$ with $q \leq 1$) or linear convergence rates depending on the strong Wolfe primal gap parameters. The exponent $q$ depends on the growth of the function around the optimum, so the function is not required to be strongly convex in the traditional sense. 
  In particular, we obtain linear rates (depending on the parameters) for non-strongly convex functions. Our rates are satisfied after a mild burn-in phase that does not depend on the target accuracy.
  We extend the results of \citep{kerdreux2019restarting} to all the known settings where versions of Frank-Wolfe enjoy linear convergence rate under strong convexity assumptions of the objective function (see Section \ref{sec:Frac_Frank_Wolfe}).
  
  \item \emph{Robust restarts.} Restart schedules often heavily depend on the value of unknown parameters. We show that because Frank-Wolfe methods naturally produce a stopping criterion in the form of the strong Wolfe gap, our restart schemes are robust and do not require knowledge of the unobserved strong Wolfe primal gap bound parameters.
  
  \item We generalize our approach in \citep{kerdreux2019restarting} to H\"older smooth functions.

\end{enumerate}

\subsection*{Outline}
In Section~\ref{sec:prelims} we briefly recall key notions and notation. We then describe our strong Wolfe primal bounds in Section~\ref{sec:LGC} and present the Fractional Away-step Frank-Wolfe Algorithm in Section~\ref{sec:sharpAWF} along with the associated restart schemes in Section~\ref{sec:restart-schemes}.
Section~\ref{sec:Holder} generalizes the analysis to H\"older smooth functions.
Section~\ref{sec:Frac_Frank_Wolfe} investigates the cases where the constraint set is not a polytope or the optimum is not necessarily on the boundary of the constraint set. They are known cases where additional structure on $f$ leads to accelerated convergence rates of the (vanilla) Frank-Wolfe algorithm.

\section{Preliminaries}
\label{sec:prelims}

Consider the following optimization problem
\BEQ\label{eq:prob}
\BA{ll}
\mbox{minimize} & f(x)\\
\mbox{subject to} & x \in \mathcal{C}
\EA\EEQ
in the variables $x\in\reals^n$, where $\mathcal{C}\subset\reals^n$ is a compact convex set and $f:\reals^n \rightarrow \reals$ is a convex function.
For the sake of simplicity, we will consider that $\mathcal{C}$ is full-dimensional.
Let $X^*$ be the set of minimizers of $f$ over $\mathcal{C}$. 
We assume that the following linear minimization oracle
\BEQ\label{eq:lmo}
\LP{\mathcal{C}}(x)\triangleq \argmin_{z\in \mathcal{C}} x^Tz
\EEQ
can be computed efficiently. By assumption here, we have $\mathcal{C}=\Co(\Ext(\mathcal{C}))$ where $\Co(\cdot)$ is the convex hull, $\Ext(\cdot)$ the set of extreme points, and Carath\'eodory's theorem shows that every point $x$ of $\mathcal{C}$ can be written as a convex combination of at most $n+1$ points in $\Ext(\mathcal{C})$ although a given representation can contain more points. We call these points the {\em support of $x$} in $\mathcal{C}$. We say that a support $\mathcal{S}$ is {\em proper} when the weights that compose the convex combination of $x$ are all positive.
\begin{definition}[Proper Support]
Consider a compact convex set $\mathcal{C}$ and $x\in\mathcal{C}$. A finite set $\mathcal{S}=\{v_i \mid i \in I \}$ with
  $v_i \in \Ext(\mathcal{C})$ for some finite index set \(I\), is a
  proper support of $x$ iff
\[
x=\sum_{i\in \mathcal{S}} \lambda_i v_i,\quad \mbox{where $\ones^T\lambda=1$ and
  $\lambda_i>0$ for all $i \in I$}.
\]
\end{definition}

We now define the \emph{strong Wolfe gap} as follows.

\begin{definition}[Strong Wolfe Gap]\label{def:strong_wolfe_gap}
Let $f$ be a smooth convex function, $\mathcal{C}$ a polytope, and let $x \in \mathcal{C}$ be arbitrary. Then the \emph{strong Wolfe gap $w(x)$ over $\mathcal{C}$} is defined as
\BEQ\label{eq:strong_wolfe_gap}
w(x) \triangleq\min_{\mathcal{S}\in\mathcal{S}_x} \max_{y \in \mathcal{S}, z \in \mathcal{C}} \nabla f(x)^T (y-z)
\EEQ
where $x \in \Co(\mathcal{S})$ and $\mathcal{S}_x = \{\mathcal{S}~|~\mathcal{S}\subset\Ext(\mathcal{C}), \mbox{is finite and $x$ a proper combination of the elements of $\mathcal{S}$}\}$, the set of proper supports of $x$.
We also write 
\[
w(x,\mathcal{S}) \triangleq \max_{y \in \mathcal{S}, z \in \mathcal{C}} \nabla f(x)^T (y-z)
\]
given $\mathcal{S}\in\mathcal{S}_x$.
\end{definition}

By construction, we have $w(x) \leq w(x,\mathcal{S})$. Note also that for
$x\in\mathcal{C}$, the quantity $w(x,\mathcal{S})$ is the sum of the Frank-Wolfe
dual gap with the away dual gap in \citep{FW-converge2015} as shows
the following decomposition
\BEQ\label{eq:decomp_strong_gap}
w(x, \mathcal{S}) =  \underbrace{\max_{y \in \mathcal{S}} \nabla f(x)^T (y-x) }_{\text{away or Wolfe
  (dual) gap}} + \underbrace{\max_{z \in \mathcal{C}} \nabla f(x)^T
(x-z)}_{\text{Frank-Wolfe (dual) gap}}.
\EEQ
Note that only $w(x, \mathcal{S})$ is observed in practice, but we use $w(x)$ to simplify the primal bounds and the convergence proof. 
Also we write the Frank-Wolfe (dual) gap as
\BEQ\label{eq:Frank_Wolfe_Dual_Gap}
g(x) \triangleq \max_{z \in \mathcal{C}} \nabla f(x)^T
(x-z).
\EEQ
We first show the following lemma on $w(x, \mathcal{S})$ and $w(x)$.
This lemma justifies the use of strong Wolfe gaps as measure of optimality in Algorithm \ref{alg:frac-AFW} and \ref{algo:scheduled_fractional}.

\begin{lemma}\label{lem:zero-gap}
Let $x\in\mathcal{C}$ and $\mathcal{S}$ be proper support of $x$. 
We have that $w(x,\mathcal{S})=0$ if and only if $x$ is an optimal solution of problem~\eqref{eq:prob}.
In particular, $w(x)=0$ if and only if $x$ is an optimal solution of problem~\eqref{eq:prob}.
Write $x^*$ and optimal solution to \eqref{eq:optim_problem}, we have
\begin{equation}
    f(x) - f(x^*) \leq w(x, \mathcal{S}).
\end{equation}
\end{lemma}
\begin{proof}\label{proof:lem_zero_gap}
We can split $w(x,\mathcal{S})$ in two parts, with
\[
w(x,\mathcal{S}) =  \max_{y \in \mathcal{S}} \nabla f(x)^T (y-x) + \max_{z \in \mathcal{C}} \nabla f(x)^T (x-z).
\]
If $x \in \mathcal{C}$, then both summands are nonnegative.
Recall $g(x) = \max_{z \in \mathcal{C}} \nabla f(x) (x-z)$ is the Wolfe gap.

Let us first assume that $x$ is an optimal solution of problem~\eqref{eq:prob} and show that $w(x,\mathcal{S})=0$.
The first order optimality conditions implies that $\nabla f(x)^T(x-v)\leq 0$ for all $v\in\mathcal{C}$. 
Since this last quantity is exactly zero when $v=x$, we have $g(x)=0$.
Besides, let $h(x)\triangleq \max_{y \in \mathcal{S}} \nabla f(x)^T (y-x)$. If $\nabla f(x)=0$ we immediately get $h(x)=0$. Suppose then $\nabla f(x)\neq 0$, since $x$ is optimal, $\nabla f(x)^T (x-v_i)\leq 0$ for all $v_i\in \mathcal{S}$ and we can write
\[
x =  \sum_{\{i:\nabla f(x)^T (x-v_i)=0\}} \lambda_i v_i + \sum_{\{i:\nabla f(x)^T (x-v_i) < 0\}} \lambda_i v_i = (1-\mu) z_1 + \mu z_2
\]
for some $0\leq\mu\leq 1$, where $\nabla f(x)^T (x-z_1)=0$ and $\nabla
f(x)^T (x-z_2)<0$. Now $0=\nabla f(x)^T (x-x)=\mu \nabla f(x)^T
(x-z_2)$ implies $\mu=0$, hence $\nabla f(x)^T (x - v_i)= 0$ for all $i\in \mathcal{S}$, so $h(x)=0$. Thus we obtain that $x$ optimal implies $w(x)=0$. 

Conversely, let us assume that $w(x,\mathcal{S})=0$. we have
\BEAS
f(x)-f^\star &\leq& \nabla f(x)^T(x-x^\star)\\
& \leq & \max_{z \in \mathcal{C}} \nabla f(x)^T (x-z)\\
& \leq & \max_{y \in S, z \in \mathcal{C}} \nabla f(x)^T (y-z)\\
& = & w(x, \mathcal{S})
\EEAS
by convexity (where $x^\star$ is any optimal solution), and the fact that $x \in \Co(\mathcal{S})$. Hence $w(x, \mathcal{S}) = 0$ implies $x$ optimal. The corollary on $w(x)$ immediately follows by construction.
\end{proof}
A function $f$ is $L$-smooth when for any $x,y\in\mathcal{C}$
\BEQ\label{eq:function_smoothness}
\|f(x) - f(y)\|_*\leq L \|x-y\|.
\EEQ
Such regularity of the gradient can also be captured via curvature.  We recall the definition of \emph{away curvature} in \citep[Appendix D]{FW-converge2015}, with
\BEQ\label{eq:away_curvature}
C_f^A\triangleq\underset{\substack{x,s,v\in\mathcal{C}\\\eta\in [0,1] \\ y=x+\eta (s-v)}}{\text{sup }}{\frac{2}{\eta^2}\big( f(y)-f(x)-\eta\langle\nabla f(x),s-v\rangle \big)},
\EEQ
where $f$ and $\mathcal{C}$ are defined in problem~\eqref{eq:prob}
above; we will use this notion of curvature for analyzing those
algorithms utilizing away steps (Algorithm \ref{alg:frac-AFW}). 
Similarly (standard) \emph{curvature} $C_f$ \citep[Appendix C]{FW-converge2015} is defined as
\BEQ\label{eq:curvature}
C_f\triangleq\underset{\substack{x,v\in\mathcal{C}\\\eta\in [0,1] \\ y=x+\eta (v-x)}}{\text{sup }}{\frac{2}{\eta^2}\big( f(y)-f(x)-\eta\langle\nabla f(x),v-x\rangle \big)},
\EEQ
and is used to bound the complexity of the classical Frank-Wolfe
method (Algorithm \ref{alg:frac-FW}).

\section{H\"{o}lderian Error Bounds}\label{sec:LGC}
We now introduce growth conditions used to bound the complexity of our variant of the Frank-Wolfe algorithm when solving the constrained optimization problem in~\eqref{eq:optim_problem}.
Let $\mathcal{C}$ be a general compact convex set with non-empty interior.
The following condition will be at the core of our complexity analysis. 

\begin{definition}[Strong Wolfe primal bound]\label{def:strong_wolfe_primal_gap}
Let $K$ be a compact neighborhood of $X^*$ in $\mathcal{C}$, where $X^*$ is the set of solutions of the constrained optimization problem \eqref{eq:optim_problem}. 
A function $f$ satisfies a $r$-strong Wolfe primal bound on $K$, if and only if there exists $r\geq 1$ and $\mu>0$ such that for all $x\in K$
\BEQ\label{eq:strong_wolfe_primal_gap}
f(x) - f^* \leq \mu w(x)^{r},
\EEQ
and $f^*$ its optimal value.
\end{definition}

In the next section, provided $f$ is a smooth convex function, we will show, for instance, that $r=2$ above guarantees linear convergence of our variant of Away Frank-Wolfe. This $2$-strong Wolfe primal bound holds notably when $f$ is strongly convex over a polytope, which corresponds to the  linear convergence bound in~\citep{FW-converge2015}, hence the following observation.

\begin{observation}[$f$ strongly convex and $\mathcal C$ a polytope]\label{obs:strong_convexity_r_2}
The results in \cite[Theorem 8 in Eq (28)]{FW-converge2015} show that when $f$ is strongly convex and $\mathcal{C}$ is a polytope then there exists $\mu_f^A>0$ such that for all $x\in\mathcal{C}$
\[
f(x)-f^*\leq \frac{w(x)^2}{2\mu_f^A}.
\]
In other words, condition~\eqref{def:strong_wolfe_primal_gap} holds with $r=2$ in this case.
\end{observation}

The fact that $w(x)=0$ if and only if $f(x)=f^*$ means that, in principle, the {\L}ojasiewicz factorization lemma \cite[\S 3.2.]{bolte2007lojasiewicz} could be used to show that condition~\eqref{eq:strong_wolfe_primal_gap} holds generically but with unobservable parameters. These parameters are inherently hard to infer because \eqref{eq:strong_wolfe_primal_gap} combines the properties of $f$ and $\mathcal{C}$, not distinguishing between the contribution of the function from that of the structure of the constrained set.
This was our initial approach in \citep{kerdreux2019restarting} but proving the subanalyticity of $w(\cdot)$ is however non-trivial.

Hence, although \eqref{eq:strong_wolfe_primal_gap} has an appealing succinct form, our results will rely on the combination of a more classical H\"{o}lderian error bound (in  Definition~\ref{def:LGC}) defined on $f$, and a \textit{scaling inequality} (defined below in Definition~\ref{def:Scaling}), essentially driven by the structure of the set $\mathcal{C}$. 
The combination of these two inequalities leads to a $r$-strong Wolfe primal bound. 
We first state the \textit{scaling inequality} relative to the strong Wolfe gap \(w(x)\) that we will use in the context of the the away step variant of the Frank-Wolfe algorithm.

\begin{definition}[$\delta$-scaling]\label{def:Scaling}
A convex set $\mathcal{C}$ satisfies a \emph{scaling inequality} if there exists $\delta(\mathcal{C})>0$ such that for all $x\in\mathcal{C}\setminus X^*$ and all differentiable convex function $f$,
\BEQ\label{eq:Scaling}\tag{Scaling}
w(x) \geq \delta(\mathcal{C}) ~ \max_{x^*\in X^*} \left \langle \nabla f(x), \, \frac{x - x^*}{\|x - x^*\|} \right \rangle.
\EEQ
\end{definition}

Here again, the strong Wolfe gap $w(x)$ is the minimum over all proper supports of $x$ of the scalar product of the (negative) gradient with the pairwise direction formed by the difference of the Frank-Wolfe vertex and the away vertex. Hence the $\delta$-scaling inequality compares the worst pairwise FW direction with the normalization of the direction $x^* - x$. Notably this condition is known to hold when $\mathcal{C}$ is a polytope, with \cite{FW-converge2015} showing the following result (see also \citep{pena2018polytope} for a simpler variant).

\begin{lemma}[\citep{FW-converge2015}]\label{lem:Scaling_polytope}
A polytope satisfies the $\delta$-scaling inequality with $\delta(\mathcal{C}) = PWidth(\mathcal{C})$, where $PWidth(\mathcal{C})$ is the pyramidal width \citep[(9)]{FW-converge2015}.
\end{lemma}

We now recall the definition of the H\"{o}lderian error bound (aka sharpness) for a function $f$ on problem \eqref{eq:optim_problem} \citep{hoffman1952approximate,law1965ensembles,lojasiewicz1993geometrie,bolte2007lojasiewicz} (see e.g., \citep{roulet2020sharpness} for more detailed references).

\begin{definition}[H\"{o}lderian error bound (HEB)]\label{def:LGC}
Consider a convex function $f$ and $K$ a compact neighborhood of $X^*$ in $\mathcal{C}$. For optimization problem \eqref{eq:optim_problem}, $f$ satisfies a $(\theta,c)$-\ref{eq:LGC} on $K$ if there exists $\theta\in [0,1/2]$ and $c>0$ such that for all $x\in K$
\BEQ\label{eq:LGC}\tag{HEB}
\min_{x^*\in X^*} \|x - x^*\| \leq c (f(x) - f^*)^\theta.
\EEQ
\end{definition}

The H\"{o}lderian error bound~\eqref{eq:LGC} locally quantifies the behavior of $f$ around the constrained optimum of problem~\eqref{eq:prob}. A similar condition was used to show improved convergence rates for unconstrained optimization in e.g., \citep{Nemi85,attouch2014variational,frankel2015splitting,karimi2016linear,bolte2017error,roulet2020sharpness,li2018calculus}.
Note that strong convexity implies $(\theta, c)$-\ref{eq:LGC} with $\theta = 1/2$ so~\eqref{eq:LGC} can be seen as a generalization of strong convexity. Here, $\theta$ will allow us to interpolate between sub-linear and linear convergence rates.

Finally, we show that when Problem \eqref{eq:optim_problem} satisfies both $\delta$-\ref{eq:Scaling} and $(\theta,c)$-\ref{eq:LGC}, the $(1-\theta)^{-1}$-strong Wolfe primal bound in~\eqref{eq:strong_wolfe_primal_gap} holds.

\begin{lemma}\label{lem:strong_wolfe_primal_gap}
Assume $f$ is a differentiable convex function satisfying $(\theta,c)$-\ref{eq:LGC} on $K$, and that $\mathcal{C}$ satisfies $\delta$-\ref{eq:Scaling} inequality. Then for all $x\in K$
\[
f(x) - f^* \leq \Big(\frac{c}{\delta}\Big)^{r} w(x)^{r},
\]
with $r=\frac{1}{1-\theta}$ and $f^*$ the objective value at constrained optima.
\end{lemma}
\begin{proof}
Assume we have $(\theta,c)$-\ref{eq:LGC} on $K$. For $x\in K\setminus X^*$, by convexity, with  $\tilde{x}\in \argmin_{x^*\in X^*}{\|x - x^*\|}$
\[
f(x) - f^*  \leq  \frac{\langle \nabla f(x), \, x - \tilde{x} \rangle}{\|x - \tilde{x}\|} \|x - \tilde{x}\|.
\]
Hence applying $(\theta,c)$-\ref{eq:LGC} leads to
\begin{eqnarray}
f(x) - f^*  &\leq &  c \frac{\langle \nabla f(x), \, x - \tilde{x} \rangle}{\|x - \tilde{x}\|} \Big(f(x) - f^*\Big)^\theta\nonumber\\
 &\leq &   c  \max_{x^*\in X^*} \frac{\langle \nabla f(x), \, x - x^* \rangle}{\|x - x^*\|} \Big(f(x) - f^*\Big)^\theta,\label{eq:identity_LGC}
\end{eqnarray}
from which we obtain
\[
f(x) - f^* \leq c^{\frac{1}{1-\theta}} \max_{x^*\in X^*} \Bigg(  \frac{\langle \nabla f(x), \, x - x^* \rangle }{\|x - x^*\|}   \Bigg)^{\frac{1}{1-\theta}}.
\]
Combining this with the $\delta$-scaling inequality, we have
\[
f(x) - f^* \leq \Big(\frac{c}{\delta}\Big)^{\frac{1}{1-\theta}} w(x)^{\frac{1}{1-\theta}},
\]
and the desired result.
\end{proof}

In the next section, varying values of $r\in[1,2[$ in \eqref{eq:strong_wolfe_primal_gap} allow to produce sub-linear complexity bounds of the form $\mathcal{O}(1/\epsilon^{1/(2-r)})$, continuously interpolating between the known sub-linear $\mathcal{O}(1/\epsilon)$ and a linear convergence rate obtained with $r=2$. 
For simplicity of exposition, we will always pick $K=\mathcal{C}$ in what follows. We also write $\textbf{Int}(\cdot)$ for the interior of a set.

\section{The Fractional Away-Step Frank-Wolfe Algorithm}\label{sec:sharpAWF}
In this section, we focus on the case where $\mathcal{C}$ is a polytope and $f$ a smooth convex function. This means, in particular, that condition~\eqref{eq:Scaling} holds.
We now state the Fractional Away-step Frank-Wolfe method as Algorithm~\ref{alg:frac-AFW}, a variant of the Away-step Frank-Wolfe algorithm, tailored for restarting.

\begin{algorithm}
  \caption{Fractional Away-step Frank-Wolfe Algorithm}
  \label{alg:frac-AFW}
  \begin{algorithmic}[1]
    \REQUIRE
      A smooth convex function $f$ with curvature~\(C_f^{A}\).
      Starting point \(x_{0} = \sum_{v\in \mathcal{S}_0}{\alpha^v_{0} v} \in \mathcal{C}\) with support $\mathcal{S}_0\subset\Ext(\mathcal{C})$.
      LP oracle \eqref{eq:lmo} and schedule parameter $\gamma > 0$.
    \STATE $t := 0$
    \WHILE{$w(x_t,\mathcal{S}_t) > e^{-\gamma} w(x_0,\mathcal{S}_0)$}
      \STATE
        \(v_t := \LP{\mathcal{C}}(\nabla f(x_{t})) \text{ and } d_t^{FW} := v_t - x_t\) \label{line:LMO_FW}
\STATE \(s_t := \LP{\mathcal{S}_t}(-\nabla f(x_{t}))\) with
        \(\mathcal{S}_t\) current active set and $d_t^{Away} := x_t - s_t$
        
        \IF{\(-\nabla f(x_t)^T d_t^{FW} >  e^{-\gamma} w(x_0,\mathcal{S}_0) / 2\)}\label{line:test_criterion}
        \STATE $d_t := d_t^{FW}$ with $\eta_{\text{max}} := 1$
        \ELSE
        \STATE $d_t := d_t^{Away}$ with $\eta_{\text{max}} := \frac{\alpha_t^{s_t}}{1-\alpha_t^{s_t}}$ \label{line:away_direction}
        \ENDIF
        \STATE $x_{t+1} := x_t + \eta_t d_t$ with $\eta_t\in [0,\eta_{\text{max}}]$ via line-search
        \STATE Update active set $\mathcal{S}_{t+1}$ and coefficients $\{\alpha_{t+1}^v\}_{v\in\mathcal{S}_{t+1}}$ \label{line:update_alpha_s}
        \STATE $t := t+1$
    \ENDWHILE
    \ENSURE $t\in\mathbb{N}$ and \(x_{t} \in \mathcal{C}\) such that $w(x_t,\mathcal{S}_t) \leq e^{-\gamma} w(x_0,\mathcal{S}_0)$
  \end{algorithmic}
\end{algorithm}

For the Away-step Frank-Wolfe or Algorithm \ref{alg:frac-AFW}, an iteration performs a \textit{drop step} when the update direction is the away direction (Line \ref{line:away_direction}) and the chosen step size $\eta_t$ is equal to $\eta_{\text{max}}={\alpha_{s_t}}/{(1-\alpha_{s_t})}$. Indeed, such an iteration removes (drops) the away vertex $s_t$ from the convex combination of $x_t$. Conversely, we will call a step a \emph{full-progress step} if it is a Frank-Wolfe step or an away step that is not a drop step.
The support $\mathcal{S}_t$ and the weights $\alpha_t$ are updated exactly as in \citep[\text{Away-step Frank-Wolfe}]{FW-converge2015}. 
Note that to perform such away versions of Frank-Wolfe, we require a linear minimization oracle over the support of the optimization iterates.
Such an oracle is typically done naively so that its cost grows linearly with the size of the support.
Algorithm \ref{alg:frac-AFW} depends on a parameter $\gamma>0$ which explicitly controls the number of iterations needed for the algorithm to stop. 
In particular, a large value of $\gamma$ will increase the number of iterations and when $\gamma$ converges to infinity, Algorithm \ref{alg:frac-AFW} tends to behave exactly like the classical Frank-Wolfe and never chooses the away direction as an update direction.
To support this intuition, we prove Appendix~\ref{sec:oneShot} that the convergence rate of one run of Fractional Away Frank-Wolfe with a large value of $\gamma$ similar to that of the classical Frank-Wolfe.

We name Algorithm \ref{alg:frac-AFW} a \textit{Fractional} version of Away Frank-Wolfe since after running the algorithm, the strong Wolfe gap $w(x_t,\mathcal{S}_t)$ (our measure of optimality) is only guaranteed to be a fraction of the initial Wolfe gap $w(x_0,\mathcal{S}_0)$.
Besides, the vanilla AFW consists in a different decision rule to decide between away-steps or FW steps (Line \ref{line:test_criterion}).

Proposition \ref{prop:max_iter_fracFW} below gives an upper bound on
the number of iterations required for Algorithm \ref{alg:frac-AFW} to
reach a given target gap $w(x_T, \mathcal{S}_T) \leq w(x_0,\mathcal{S}_0) e^{-\gamma}$. The assumption $e^{-\gamma} w(x_0,\mathcal{S}_0) / 2 \leq C_f^{A}$ in this
proposition measures the complexity of a burn-in phase whose cost is marginal as shown in Proposition~\ref{obs:burn_in_phase}.

\begin{proposition}[Fractional Away-Step Frank-Wolfe Complexity]
\label{prop:max_iter_fracFW}
Let $f$ be a smooth convex function with away curvature $C_f^{A}$ such that the $r$-strong Wolfe primal bound in~\eqref{eq:strong_wolfe_primal_gap} holds on $\mathcal{C}$ (with $1\leq r\leq 2$ and $\mu>0$). Let $\gamma > 0$ and assume $x_0\in \mathcal{C}$ is such that $e^{-\gamma} w(x_0) / 2 \leq C_f^{A}$. Algorithm~\ref{alg:frac-AFW} outputs an iterate $x_T\in \mathcal{C}$ such that
\[
  w(x_T,\mathcal{S}_T) \leq w(x_0, \mathcal{S}_0) e^{-\gamma}
\]
after at most
\[
  T \leq |\mathcal{S}_0| - |\mathcal{S}_T| + 16 e^{2 \gamma} C_f^{A} \mu w(x_0, \mathcal{S}_0)^{r-2}
\]
iterations, where $\mathcal{S}_0$ and $\mathcal{S}_T$ are respectively the supports of $x_0$ and $x_T$.
\end{proposition}
\begin{proof}
Let us write $w_0\triangleq w(x_0,\mathcal{S}_0)$ to simplify notation.
Because of the test criterion in Line \ref{line:test_criterion}, one can lower bound the inner product between the update directions $d_t$ and the negative gradients $ -\nabla f(x_t)$ of the form
\begin{equation}\label{eq:direction_minimal_improvement}
-\nabla f(x_t)^T d_t > e^{-\gamma} w(x_0, \mathcal{S}_0)/2~.
\end{equation}
Indeed, this holds by definition when $d_t = d_t^{FW}$.
Otherwise, $d_t=d_t^{Away}$ and $-\nabla f(x_t)^T d_t^{FW} < e^{-\gamma} w_0/2$. 
Also because the algorithm has not terminated yet, we have $w(x_t, \mathcal{S}_t) > e^{-\gamma} w_0$. 
The decomposition of the strong Wolfe gap \eqref{eq:decomp_strong_gap} then yields
\[
w(x_t,\mathcal{S}_t) = -\nabla f(x_t) ^T d_t^{FW}  - \nabla f(x_t)^Td_t^{Away}  >  e^{-\gamma} w_0,
\]
so that we indeed obtain \eqref{eq:direction_minimal_improvement}
\[
 -\nabla f(x_t)^Td_t^{Away}  >  e^{-\gamma} w_0 +  \nabla f(x_t)^T d_t^{FW} \geq e^{-\gamma} w_0 - e^{-\gamma} w_0/2 = e^{-\gamma} w_0/2.
\]
Using curvature in~\eqref{eq:away_curvature}, we have for $d_t$,
\[
f(x_t +\eta d_t) \leq  f(x_t) + \eta \nabla f(x_t)^T d_t + \frac{\eta^2}{2}C_f^{A}~,
\]
which implies
\[
f(x_t) - f(x_t +\eta d_t) \geq  \eta -\nabla f(x_t)^Td_t -  \frac{\eta^2}{2}C_f^{A}~.
\]
We can lower bound progress $f(x_t) - f(x_{t+1})$ with $x_{t+1} = x_t +\eta d_t$ at each iteration for full-progress steps. 
Indeed, for Frank-Wolfe steps,
\BEAS
f(x_{t}) - f(x_{t+1}) &\geq & \underset{\eta\in[0,1]}{\text{max }} \Big\{\eta (-\nabla f(x_t))^Td_t  -  \frac{\eta^2}{2}C_f^{A} \Big\} \\
&\geq & \underset{\eta\in[0,1]}{\text{max }} \Big\{ \eta e^{-\gamma}w_0/2 -  \frac{\eta^2}{2}C_f^{A} \Big\}
\EEAS
Hence because of exact line-search, assuming $e^{-\gamma}w_0/2 \leq C_f^{A}$ holds, we obtain
\BEQ\label{eq:curvature_lower_bound}
f(x_{t}) - f(x_{t+1})  \geq  \frac{ w_0^2}{8 C_f^{A} e^{2\gamma}}.
\EEQ
For all away steps, we have
\[
f(x_t) - f(x_t + \eta d_t) \geq \underset{\eta\in[0,\eta_{\text{max}}]}{\text{max }} \Big\{\eta e^{-\gamma}w_0/2 -  \frac{\eta^2}{2}C_f^{A} \Big\}.
\]
Yet for Away steps that are not drop steps, assuming $e^{-\gamma}w_0/2 \leq C_f^{A}$ holds, the minimal $\eta^*$ is such that $0 < \eta^* < \eta_{\text{max}}$, and the same conclusion as in \eqref{eq:curvature_lower_bound} for Frank-Wolfe steps follows.

Write $T = T_d + T_f$ the number of iterations for Algorithm \ref{alg:frac-AFW} to finish, where $T_d$ denotes the number of drop steps, while $T_f$ stands for the number of full-progress steps. Hence we have,
\BEAS
f(x_0) - f(x_T) &=& \sum_{t=0}^{T-1}{f(x_t) - f(x_{t+1})}\\
& \geq & T_f \frac{w_0^2}{8 C_f^{A} e^{2\gamma}}.
\EEAS
Because $f$ satisfies a $r$-strong Wolfe primal gap on $\mathcal{C}$ we have, when $x_0\in\mathcal{C}$,
\begin{equation}\label{eq:use_of_wolfe_bounds}
f(x_0) - f(x_T) \leq f(x_0) - f^* \leq \mu w(x_0)^{r} \leq \mu w(x_0, \mathcal{S}_0)^r,
\end{equation}
by definition of $w(x)$. 
We then get an upper bound on the number $T_f$ of full-progress steps
\[
T_f \leq 8 C_f^{A} e^{2\gamma} \mu w_0^{r-2}.
\]
Finally writing $|\mathcal{S}_0|$ (resp. $|\mathcal{S}_T|$) the size of the support of $x_0$ (resp. $x_T$), and $T_{FW}$ the number of Frank-Wolfe steps which add a new vertex to an iterate of the Fractional Away-step Frank-Wolfe Algorithm. We have that $T_{FW}\leq T_f$ and the size of the support $\mathcal{S}_t$ of $x_t$ satisfies $|\mathcal{S}_0| - T_d + T_{FW} = |\mathcal{S}_T|$ hence
\[
|\mathcal{S}_0| - |\mathcal{S}_T| + T_f \geq T_d,
\]
and we finally obtain
$
T \leq |\mathcal{S}_0| - |\mathcal{S}_T| + 16 C_f^{A} e^{2\gamma} \mu w_0^{r-2}.
$
\end{proof}

The following observation shows that the assumption $e^{-\gamma} w(x_0, \mathcal{S}_0) / 2 \leq C_f^{A}$ in Proposition \ref{prop:max_iter_fracFW} has a marginal impact on complexity.

\begin{proposition}[Burn-in phase]\label{obs:burn_in_phase}
After at most
\[
8 e^{\gamma} \Big\lfloor \frac{1}{\gamma} \ln{\frac{w_0}{2C_f^{A}}}\Big\rfloor + |\mathcal{S}_0|,
\]
cumulative iterations of Algorithm \ref{alg:frac-AFW}, with constant schedule parameter $\gamma > 0$, we obtain a point $x\in\mathcal{C}$ such that $e^{-\gamma} w(x, \mathcal{S}_x) / 2 \leq
C_f^{A}$.
\end{proposition}

\begin{proof}
The proof closely follows that of Proposition \ref{prop:max_iter_fracFW} as well as \citep[Appendix D]{FW-converge2015}.
Let $w_0= w(x_0, \mathcal{S}_0)$. Suppose that \(e^{-\gamma} w_0 / 2 > C_f^{A}\) and note that the lower bound \eqref{eq:direction_minimal_improvement} holds similarly. Let us consider the progress incurred with \textit{full progress steps}. Recall that the curvature of $f$ with the line-search ensure that for any $\eta\in[0,\eta_{\text{max}}]$, we have
\[
f(x_t) - f(x_{t+1}) \geq \eta (-\nabla f(x_t))^T d_t - \frac{\eta^2}{2} C_f^{A}.
\]
For a Frank-Wolfe step or an away step with $\eta_{\text{max}}\geq 1$, we obtain by choosing $\eta=1$,
\[
f(x_t) - f(x_{t+1}) \geq (-\nabla f(x_t))^T d_t - C_f^{A}/2 \geq e^{-\gamma} w_0/2 - e^{-\gamma} w_0/4 = e^{-\gamma} w_0/4.
\]
The last possibility is that the full progress step is an away step with $\eta_{\text{max}}<1$. Since it is not a drop step, we have $\eta_t<\eta_{\text{max}}$. In particular then $\eta_t$ is a local minimum of the convex function $f(x_t + \gamma d_t)$ and hence $\underset{\gamma\in[0,\gamma_{\text{max}}]}{\text{min }} f(x_t + \gamma d_t) = \underset{\gamma}{\text{min }} f(x_t + \gamma d_t)$. With $\eta=1$, we obtain
\[
f(x_t) - f(x_{t+1}) \geq (-\nabla f(x_t))^T d_t - C_f^{A}/2.
\]
Finally, we conclude that for any full progress step we have 
\[
f(x_t) - f(x_{t+1}) \geq e^{-\gamma}w_0 /4.
\]
Moreover, the strong Wolfe gap is an upper bound on the primal gap, \textit{i.e.} $f(x_0) - f(x^*) \leq w_0$.
Write $T$ the number of iterations Algorithm~\ref{alg:frac-AFW} performs. As in Proposition~\ref{prop:max_iter_fracFW}, note $T_f$ the number of full progress steps. Similarly, we obtain
\[
T_f e^{-\gamma} w_0/4  \leq  \sum_{t=0}^{T-1}{f(x_t) - f(x_{t+1})} =   f(x_0) - f(x_T) \leq f(x_0) - f(x^*) \leq w_0.
\]
Hence
\[
T_f e^{-\gamma} w_0/4 \leq  w_0
\]
and $T_f \leq 4 e^{\gamma}$. Also
\[
T = T_d + T_f \leq 2 T_f + |\mathcal{S}_0| - |\mathcal{S}_T|~,
\]
so that
\[
T  \leq  8 e^{\gamma} + |\mathcal{S}_0| - |\mathcal{S}_T|.
\]
In other words, when $e^{-\gamma} w(x_0)/2 < C_f^{A}$, after at most $8 e^{\gamma} + |\mathcal{S}_0|$ iterations, the Fractional Away-step Frank-Wolfe terminates, with an iterate $x_T$ which strong Wolfe gap is guaranteed to be a fraction of the initial one, \textit{i.e.}, $w(x_T, \mathcal{S}_T)\leq e^{-\gamma} w(x_0, \mathcal{S}_0)$.

Then, consider running the Fractional Away-step Frank-Wolfe $N$ times, initializing each run with the output of the previous run. Write $x_{T_i}$ each output of the $i^{th}$ run of Algorithm \ref{alg:frac-AFW}. After N runs, $x_{T_N}$ satisfies $w(x_{T_N}, \mathcal{S}_{T_N}) \leq e^{-\gamma N} w_0$. Hence, if $N$ satisfies
\[
e^{-\gamma (N+1)} w_0 /2 \leq C_f^A,
\]
then $x_{T_N}$ verifies $e^{-\gamma} w(x_{T_N}, \mathcal{S}_{T_N}) \leq C_f^A$. In particular, it is sufficient to chose $ N = \big\lfloor \frac{1}{\gamma} \ln{w_0/(2C_f^{A})}\big\rfloor$. 

Finally, since the $i^{th}$ run of Fractional Away-step Frank-Wolfe performs at most $8e^{\gamma} + |\mathcal{S}_{T_{i-1}}| - |\mathcal{S}_{T_i}|$ iterations, to ensure the burn-in phase condition, we need at most
\[
\sum_{i=1}^N 8 e^{\gamma} + |\mathcal{S}_{T_{i-1}}| - |\mathcal{S}_{T_i}| \leq 8 e^{\gamma} \Big\lfloor \frac{1}{\gamma} \ln{\frac{w_0}{2C_f^{A}}}\Big\rfloor + |\mathcal{S}_0|~~\text{ iterations.}
\]
\end{proof}

\section{Restart Schemes}\label{sec:restart-schemes}
Consider a point $x_{k-1}$ with strong Wolfe gap $w(x_{k-1},\mathcal{S}_{k-1})$. Algorithm~\ref{alg:frac-AFW} with parameter $\gamma_k>0$, outputs a point $x_k$ and we write
\[
x_{k} \triangleq \mathcal{F}(x_{k-1},w(x_{k-1},\mathcal{S}_{k-1}),\gamma_k).
\]
Following \citep{roulet2020sharpness} we define \textit{scheduled restarts} for Algorithm \ref{alg:frac-AFW} as follows.
\begin{algorithm}[H]
	\caption{Scheduled restarts for Fractional Away-step Frank-Wolfe \label{algo:scheduled_fractional}}
	\begin{algorithmic}
		\STATE{\textbf{Input:} $\tilde{x}_0\in\reals^n$ and a sequence $(\gamma_k)>0$ and $\epsilon > 0$ and $T\in\mathbb{N}$.}
		\STATE $t:=0$
		\WHILE{$w(x_{k-1}, \mathcal{S}_{k-1})> \epsilon $ and $t < T$}
		\STATE
		\vspace*{-0.3cm}
		\[
		\textstyle
		(x_k, \mathcal{S}_k, T_k) := \mathcal{F}(x_{k-1},w(x_{k-1},\mathcal{S}_{k-1}),\gamma_k) \text{ and } t:= t + T_k.
		\]
		\vspace*{-0.5cm}
		\ENDWHILE
		\STATE{\textbf{Output:} $x_k$}
	\end{algorithmic}
\end{algorithm}

\noindent Note that one overall burn-in phase is sufficient to ensure the condition $e^{-\gamma_i}w(x_{i-1},\mathcal{S}_{i-1})/2 \leq C_f^{A}$ at each restart.

Algorithm \ref{algo:scheduled_fractional} is similar to the restart scheme in \citep[Section 4]{roulet2020sharpness} where a termination criterion is available. In this situation, \citep{roulet2020sharpness} show that the convergence rate of restarted gradient methods is robust to a suboptimal choice of restart scheme parameter $\gamma$.
Here, we also show that our restart scheme is adaptive to the unknown parameters in $(\theta, c)$-\ref{eq:LGC}. 

Note that Algorithm \ref{algo:scheduled_fractional} shares a similar structure with the methods in \citep{lan2017conditional,braun2019blended}.
We will see below in Proposition~\ref{prop:robustness_choice_gamma} that tuning $\gamma$ only has a marginal impact on the complexity bound. Note also that when $\theta \in [0,1/2]$, the condition interpolates between the non-strongly convex function $f$ and a strongly convex function scenarios. 
For the sake of clarity, our convergence results depend on a burn-in phase condition on the initial strong Wolfe gap, \textit{i.e.} $e^{-\gamma}w_0/2 \leq C_f^A$. Proposition \ref{obs:burn_in_phase} shows that it is satisfied after an initial linear convergence regime.

\begin{theorem}[Rate for constant restart schemes]\label{th:convergence_restart_schemes_linear}
Let $f$ be a smooth convex function with away curvature $C_f^{A}$.
Assume $\mathcal{C}$ satisfies $\delta$-\ref{eq:Scaling} and $f$ is $(\theta, c)$-\ref{eq:LGC} on $\mathcal{C}$.
Let $\gamma > 0$ and assume $x_0\in \mathcal{C}$ is such that $e^{-\gamma} w(x_0,\mathcal{S}_0) / 2 \leq C_f^{A}$ (see, Proposition \ref{obs:burn_in_phase}).
With $\gamma_k = \gamma$, the output of Algorithm \ref{algo:scheduled_fractional} satisfies ($r= \frac{1}{1-\theta}$)
\begin{equation}\label{eq:cv_rates}
  \left\{
    \begin{split}
    f(x_T) - f^* &\leq w_0\frac{1}{\Big( 1 + \tilde{T} C_{\gamma}^r \Big)^{\frac{1}{2-r}}}~~\text{ when } 1 \leq r < 2\\
    f(x_T) - f^* &\leq  w_0 \exp\left(-\frac{\gamma}{e^{2\gamma}}\frac{\tilde{T}}{16 C_f^{A} \mu }\right) ~~\text{ when } r=2,
    \end{split}
  \right.
\end{equation}
where $T$ is the cumulative number of Linear Minimization Oracle calls (Line \ref{line:LMO_FW} in Algorithm \ref{alg:frac-AFW}) in Algorithm \ref{algo:scheduled_fractional}, with $w_0=w(x_0,\mathcal{S}_0)$, $\tilde{T}\triangleq T - (|\mathcal{S}_0| - |\mathcal{S}_T|)$, and
\BEQ\label{eq:fct_gamma_rate}
C_{\gamma}^r \triangleq \frac{e^{\gamma (2-r)}-1}{ 16 C_f^{A} \mu e^{2\gamma} w(x_0,\mathcal{S}_0)^{r-2}},
\EEQ
with $\mu=\frac{c}{\delta}$.
\end{theorem}

\begin{proof}
Denote by $R$ the number of restarts in Algorithm \ref{alg:frac-AFW} for $T$ total iterations.
By design
\begin{equation}\label{eq:last_iterate_restart_scheme}
w(x_R, \mathcal{S}_R) \leq w_0 e^{-\gamma R}.
\end{equation}
Because $f$ is $(\theta, c)$-\ref{eq:LGC} and $\mathcal{C}$ satisfies $\delta$-\ref{eq:Scaling}, via Lemma \ref{lem:strong_wolfe_primal_gap}, $f$ satisfies the $r$-strong Wolfe primal bound \eqref{eq:strong_wolfe_primal_gap} with $r = \frac{1}{1-\theta}$.
Using Proposition \ref{prop:max_iter_fracFW} and repeatedly \eqref{eq:last_iterate_restart_scheme}, the total number $T$ of steps of Algorithms \ref{alg:frac-AFW} is upper-bounded by
\[
T \leq |\mathcal{S}_0| - |\mathcal{S}_T| + 16 C_f^{A} \mu e^{2\gamma} w_0^{r-2}\sum_{i=0}^{R-1}{e^{-\gamma i (r-2)}}~.
\]
Let us now distinct the case and first suppose that $1\leq r<2$.
We have the following upper bound on~$T$,
\[
T \leq |\mathcal{S}_0| - |\mathcal{S}_T| +  16 C_f^{A} \mu e^{2\gamma} w_0^{r-2} \frac{e^{\gamma (2-r)R}-1}{e^{\gamma (2-r)}-1} = |\mathcal{S}_0| - |\mathcal{S}_T| + \frac{e^{\gamma(2-r)R}-1}{C_{\gamma}^r},
\]
hence, with $\tilde{T}= T - |\mathcal{S}_0| + |\mathcal{S}_T|$, 
\[
e^{-\gamma R}\leq \frac{1}{\Big( 1 + \tilde{T} C_{\gamma}^r \Big)^{\frac{1}{2-r}}}.
\]
Thus, for $1 \leq r < 2$,
\[
w(x_R,\mathcal{S}_R) \leq w_0\frac{1}{\Big( 1 + \tilde{T} C_{\gamma}^r \Big)^{\frac{1}{2-r}}}.
\]
Now, the remaining case $r=2$ leads to
\[
T \leq  |\mathcal{S}_0| - |\mathcal{S}_T| + 16 C_f^{A} \mu e^{2\gamma} R,
\]
and hence
\[
w(x_R, \mathcal{S}_R) \leq  w_0 \exp\left(-\gamma\frac{\tilde{T}}{16 C_f^{A} \mu e^{2\gamma}}\right),
\]
which yields the desired result.
\end{proof}

\begin{corollary}
When $\mathcal{C}$ is a polytope and $f$ a smooth convex function satisfying $(\theta, c)$-\ref{eq:LGC}, rates in Theorem \ref{th:convergence_restart_schemes_linear} hold. In particular when $f$ is strongly convex, $\theta = \frac{1}{2}$ (and hence $r=2$) and Algorithm \ref{algo:scheduled_fractional} converges linearly. When $f$ is simply smooth, $\theta = 0$ (and hence $r=1$) and Algorithm \ref{algo:scheduled_fractional} converges sub-linearly with a rate of $\mathcal{O}(1/T)$.
\end{corollary}
Note also that for $r \rightarrow 2$, we recover the same complexity rates as for $r=2$
\[
\underset{r\rightarrow 2}{\lim}~~ \frac{1}{\Big( 1 + \tilde{T} C_{\gamma}^r \Big)^{\frac{1}{2-r}}} = \exp\Big(-\frac{\gamma}{e^{2\gamma}}\frac{\tilde{T}}{8 C_f^{A} \mu }\Big)~.
\]

\noindent The complexity bounds in Theorem \ref{th:convergence_restart_schemes_linear} depend on $\gamma$, which controls the convergence rate.
Optimal choices of $\gamma$ depend on $r$, a constant that we generally do not know nor observe.
However, in the following we show that simply picking $\gamma = 1/2$ leads to optimal complexity bounds up to a constant factor.
In fact, picking a constant gamma (independent of $r$) we also recover a simple version of \citep[Algorithm 1]{braun2019blended} (without the cheaper Weak Separation Oracle that replaces the Linear Minimization Oracle).

\begin{proposition}[Robustness in $\gamma$]\label{prop:robustness_choice_gamma}
Suppose $f$ satisfies the $r$-strong Wolfe primal bound~\eqref{eq:strong_wolfe_primal_gap} with $r>0$.
Write $\gamma^*(r)$ as the optimal choice of $\gamma>0$ in the coarser complexity bounds \eqref{eq:cv_rates} of Theorem~\ref{th:convergence_restart_schemes_linear} where $\tilde{T}$ is lower bounded by $\bar{T}=T - |\mathcal{S}_0|$.
Consider running Algorithm~\ref{algo:scheduled_fractional} with $\gamma = 1/2$ and the same assumptions as in Theorem~\ref{th:convergence_restart_schemes_linear}, the output $\hat{x}$ satisfies
\[
h(\hat{x})\leq \sqrt{\frac{e}{4\big(\sqrt{e} - 1\big)}} w_0\frac{1}{\Big( 1 + \bar{T} C_{\gamma^*(r)}^r \Big)^{\frac{1}{2-r}}}~~\text{ when } 1\leq r<2,
\]
where
\[
C_{\gamma}^r = \frac{e^{\gamma (2-r)}-1}{ 16 e^{2\gamma} C_f^{A} \mu w(x_0,\mathcal{S}_0)^{r-2}},
\]
as in \eqref{eq:fct_gamma_rate}. When $r=2$, we have $\gamma^*(r) = 1/2$.
\end{proposition}
\begin{proof}
When $1 \leq r < 2$, from Theorem \ref{th:convergence_restart_schemes_linear} we have
\BEQ\label{eq:bound_r_not_2}
 f(x_T) - f^* \leq w_0\frac{1}{\Big( 1 + \bar{T} C_{\gamma}^r \Big)^{\frac{1}{2-r}}}.
\EEQ
With definition of $C_{\gamma}^r$ in \eqref{eq:fct_gamma_rate}, 
minimizing \eqref{eq:bound_r_not_2} is equivalent to maximizing (for $\gamma > 0$)
\[
B(\gamma) = \Big(\frac{e^{\gamma (2-r)}-1}{e^{2\gamma}}\Big).
\]
Hence the optimum schedule parameter $\gamma^*(r)$ is
\[
\gamma^*(r) = \frac{\ln(2)-\ln(r)}{2-r}~~\text{when }1 \leq r < 2.
\]
In particular $\gamma^*(r)\in ]1/2;\text{ln}(2)]$. Let's now show that the bound in \eqref{eq:bound_r_not_2} obtained with the optimal $\gamma^*(r)$ is comparable to the bound obtained with $\gamma = \frac{1}{2}$. The function 
\[
H(r) = \frac{     \Big( 1 + \bar{T} C_{\gamma^*(r)}^r \Big)^{\frac{1}{2-r}}      }{     \Big( 1 + \bar{T} C_{1/2}^r \Big)^{\frac{1}{2-r}}}
\]
is decreasing in $r$. Write $\tilde{C} \triangleq 8 C_f^{A}\mu w(x_0,\mathcal{S}_0)$, we have $C^1_{\gamma^*(1)} = 1/(4\tilde{C})$ and $C^{1}_{1/2} = \frac{\sqrt{e} - 1}{e}/\tilde{C}$ hence
\[
H(1) = \sqrt{ \frac{1 + \frac{\bar{T}}{\tilde{C}} \frac{1}{4}}{1 + \frac{\bar{T}}{\tilde{C}} \frac{\sqrt{e}-1}{e} }} \leq \sqrt{\frac{e}{4\big(\sqrt{e} - 1\big)}}~.
\]
Hence, with $H(1)\geq H(r)$, we get for any $r\in[1, 2[$ 
\[
\frac{ 1 }{ \Big( 1 + \bar{T} C_{1/2}^r \Big)^{\frac{1}{2-r}}} \leq   \sqrt{\frac{e}{4\big(\sqrt{e} - 1\big)}}
\frac{1}{\Big( 1 + \bar{T} C_{\gamma^*(r)}^r \Big)^{\frac{1}{2-r}}}.
\]
When $r=2$, the optimal choice for $\gamma$ is $1/2$, maximizing the function ${\gamma}/{e^{2\gamma}}$.
\end{proof}

In Figure \ref{fig:exp_multiple_gamma}, we illustrate the convergence behavior of Algorithm \ref{algo:scheduled_fractional} along with that of the Away Frank-Wolfe. 
The algorithms have a similar behavior in the primal gap $f(x_t) - f(x^*)$.
For numerically competitive \textit{corrective} versions of Frank-Wolfe, see, e.g., \citep{garber2016linear,bashiri2017decomposition,combettes2020boosting} and references therein.

\begin{figure*}[h!]
	\begin{center}
		\begin{subfigure}[t]{0.35\linewidth}
			\includegraphics[width=\linewidth]{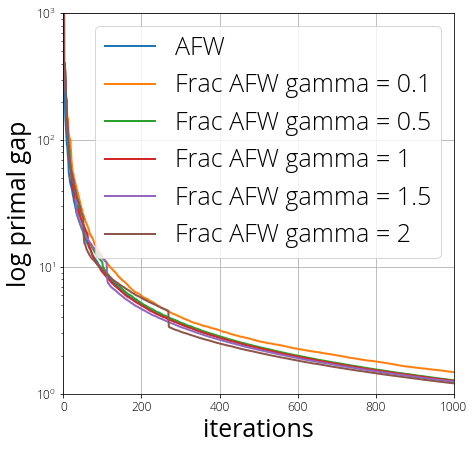}
		\end{subfigure}
		\begin{subfigure}[t]{0.35\linewidth}
			\includegraphics[width=\linewidth]{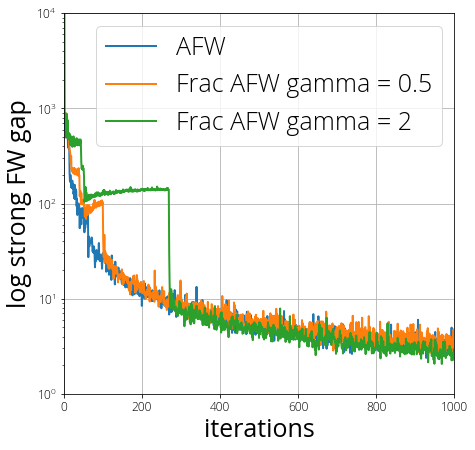}
		\end{subfigure}
		\caption{Representative examples on Lasso with various values of $\gamma$ in restart schemes of Algorithm \ref{alg:frac-AFW}.}
		\label{fig:exp_multiple_gamma}
	\end{center}
\end{figure*}

\section{Analysis under H\"older Smoothness}\label{sec:Holder}

In the following we generalize our results on convergence rates using a refined
regularity assumption on $f$. A differentiable function $f$ is
$(L,s)$-H\"older smooth on $\mathcal{C}$ when
\[ 
\|\nabla f(x) - \nabla f(y)\|_2 \leq
L \|x - y\|_2^{s-1}, \quad \mbox{for $x,y\in\mathcal{C},$}
\]
with $s\in]1,2]$. H\"older smoothness interpolates between non-smooth ($s=1$) and smooth ($s=2$) assumptions. We write the analog of the away curvature \eqref{eq:away_curvature} for $(L,s)$-H\"older smooth functions as
\[
C_{f,s}^{A} \triangleq\underset{\substack{x,u,v\in\mathcal{C}\\\eta\in [0,1] \\
    y=x+\eta (u-v)}}{\text{sup }}{\frac{s}{\eta^s}\big(
  f(y)-f(x)-\eta\langle\nabla f(x),u-v\rangle \big)}.  
\]
Note that as in \eqref{eq:away_curvature}, $f$ needs to be defined on the Minkowski sum $\mathcal{C}^A$. Let us now provide equivalent results for the complexity of Fractional Away-step Frank-Wolfe algorithm and the complexity bound of the constant restart scheme with $(L,s)$-H\"older smooth functions.

\begin{proposition}[H\"older Smooth Complexity]
\label{prop:max_iter_fracFW_holder}
Let $f$ be a $(L,s)$-H\"older smooth convex function with away curvature \(C_{f,s}^{A}\) such that the $r$-strong Wolfe primal bound in~\eqref{eq:strong_wolfe_primal_gap} holds on $\mathcal{C}$ with $\mu>0$. Let $\gamma > 0$ and assume $x_0\in \mathcal{C}$ is such that $e^{-\gamma} w(x_0, \mathcal{S}_0) / 2 \leq C_{f,s}^{A}$. Algorithm~\ref{alg:frac-AFW} outputs an iterate $x_T\in \mathcal{C}$ such that
\[
  w(x_T,\mathcal{S}_T) \leq w(x_0, \mathcal{S}_0) e^{-\gamma}
\]
after at most (with $r=\frac{1}{1-\theta}$)
\[
  T \leq |\mathcal{S}_0| - |\mathcal{S}_T| + 2^{1+\frac{s}{s-1}} \frac{s}{s-1} e^{\frac{s}{s-1}\gamma} \mu \big(C_{f,s}^{A}\big)^{\frac{1}{s-1}} w(x_0, \mathcal{S}_0)^{r-\frac{s}{s-1}}
\]
iterations, where $\mathcal{S}_0$ and $\mathcal{S}_T$ are the supports of respectively $x_0$ and $x_T$.
\end{proposition}

\begin{proof}
The proof is very similar to that required for smooth-functions, so we only detail key points. The update direction satisfies
\[
(-\nabla f(x_t))^T d_t > e^{-\gamma} w_0/2~.
\]
Applying the definition of the H\"older curvature
\[
f(x_t) - f(x_t + \eta d_t) \geq \underset{\eta\in[0,\eta_{\text{max}}]}{\text{max }} \{ \eta e^{-\gamma} w_0/2 - \frac{\eta^s}{s} C_{f,s}^{A}\} = \underset{\eta\in[0,\eta_{\text{max}}]}{\text{max }} g(\eta)~.
\]
The unconstrained maximum of $g$ is reached at $\eta^* = \Big(e^{-\gamma} \frac{w_0}{2C_{f,s}^{A}}\Big)^{\frac{1}{s-1}}$. Hence with the burn-in phase hypothesis, we guarantee $\eta^*\leq 1$. With classical arguments, for all non-drop steps, the progress in the objective function value is lower bounded by
\[
f(x_t) - f(x_t + \eta d_t) \geq \frac{1}{\big(C_{f,s}^{A}\big)^{\frac{1}{s-1}}} \frac{s-1}{s} 2^{-\frac{s}{s-1}} e^{-\gamma \frac{s}{s-1}} w_0^{\frac{s}{s-1}}.
\]
It finally follows that
\[
T \leq 2 \mu w_0^{r-\frac{s}{s-1}} 2^{\frac{s}{s-1}} \frac{s}{s-1} e^{\gamma \frac{s}{s-1}} + |\mathcal{S}_0| - |\mathcal{S}_T|~
\]
which is the desired bound.
\end{proof}

We are ready to establish the convergence rates of our restart scheme in the H\"older smooth case.

\begin{theorem}[H\"older rate for constant restart schemes]\label{th:convergence_restart_schemes_linear_holder}
Let $f$ be a $(L,s)$-H\"older smooth convex function with H\"older curvature \(C_{f,s}^{A}\), satisfying $(\theta,c)$-\ref{eq:LGC} on $\mathcal{C}$, and $\mathcal{C}$ satisfying a $\delta$-\ref{eq:Scaling} inequality. Let $\gamma > 0$ and assume $x_0\in K$ is such that $e^{-\gamma} w(x_0,\mathcal{S}_0) / 2 \leq C_{f,s}^{A}$. With $\gamma_k = \gamma$, the output of Algorithm \ref{algo:scheduled_fractional} satisfies
\begin{equation*}
    \begin{split}
    f(x_T) - f^* &\leq w_0\frac{1}{\Big( 1 + \tilde{T} C_{\gamma}^\tau \Big)^{\frac{1}{\tau}}}~~\text{ when } 1 \leq  r < \frac{s}{s-1}
    \end{split}
\end{equation*}
after $T$ steps, with $w_0 \triangleq w(x_0,\mathcal{S}_0)$, $\tilde{T}\triangleq T - (|\mathcal{S}_0| - |\mathcal{S}_T|)$, and $\tau \triangleq \frac{s}{s-1} - r$. Also
\[
C_{\gamma}^{\tau} \triangleq \frac{e^{\gamma \tau}-1}{ C_s e^{\frac{s}{s-1}\gamma} w(x_0)^\tau},
\]
with $C_s \triangleq 2^{1+\frac{s}{s-1}} \frac{s}{s-1} \frac{c}{\delta} \big(C_{f,s}^{A}\big)^{\frac{1}{s-1}}$.
\end{theorem}

\begin{proof}
Denote by $R$ the number of restarts after $T$ total inner iterations. We get
\[
T \leq \sum_{i=0}^{R-1}{ |\mathcal{S}_i| - |\mathcal{S}_{i+1}|  + 2^{1+\frac{s}{s-1}} \frac{s}{s-1} e^{\frac{s}{s-1}\gamma} \big(C_{f,s}^{A}\big)^{\frac{1}{s-1}} \mu w(x_i, \mathcal{S}_i)^{r-\frac{s}{s-1}}}~.
\]
Since $w(x_i, \mathcal{S}_i)\leq w_0 e^{-\gamma i}$, it follows that
\[
T \leq |\mathcal{S}_0| - |\mathcal{S}_T| +  2^{1+\frac{s}{s-1}} \frac{s}{s-1} e^{\frac{s}{s-1}\gamma} \big(C_{f,s}^{A}\big)^{\frac{1}{s-1}} \mu w_0^{r-\frac{s}{s-1}}  \sum_{i=0}^{R-1} e^{-\gamma i (r-\frac{s}{s-1}) }~.
\]
Write $C_s = 2^{1+\frac{s}{s-1}} \frac{s}{s-1} \big(C_{f,s}^{A}\big)^{\frac{1}{s-1}} \mu$ and $\tau = \frac{s}{s-1} - r$ we have
\[
T \leq |\mathcal{S}_0| - |\mathcal{S}_T| + C_s e^{\frac{s}{s-1}\gamma} w_0^{r-\frac{s}{s-1}}  \frac{e^{\gamma R \tau}-1}{e^{\gamma \tau}-1},
\]
it follows that 
\[
e^{-\gamma R} \leq \frac{1}{\big(  1 + (T-(|\mathcal{S}_0|-|\mathcal{S}_T|)) \frac{(e^{\gamma\tau}-1)}{C_s e^{\frac{s}{s-1} \gamma} w(x_0)^\tau}  \big)^{\frac{1}{\tau}}}.
\]
which yields the desired result.
\end{proof}

Note that $r < \frac{s}{s-1}$ is always ensured because
$s\in]1,2]$. In particular we only get linear convergence when
$r=s=2$ as for gradient methods \citep{roulet2020sharpness}. We now show, as in Proposition~\ref{obs:burn_in_phase}, that the assumption $e^{-\gamma} w(x_0,\mathcal{S}_0) / 2 \leq C_f^{A}$ has a marginal impact on complexity when the function is $(L,s)$-H\"older smooth.

\begin{proposition}[Burn-in phase for H\"older smooth functions]\label{obs:burn_in_phase_holder}
After at most
\[
4 \frac{s}{s-1} \frac{e^{\gamma}}{\gamma} \ln\big(\frac{w_0}{2 C_{f,s}^{A}}\big) + |\mathcal{S}_0|
\]
cumulative iterations of Algorithm \ref{alg:frac-AFW}, with constant schedule parameter $\gamma > 0$, we get a point $x$ such that $e^{-\gamma} w(x, \mathcal{S}) / 2 \leq
C_{f,s}^{A}$ when $f$ is $(L,s)$-H\"older smooth with $s>1$.
\end{proposition}

\begin{proof}
  Assume we have $e^{-\gamma}w_0/2 > C_{f,s}^{A}$. Classically, the
  curvature argument ensures that we have for non-drop steps \BEAS
  f(x_t) - f(x_{t+1}) &\geq & \eta_t e^{-\gamma}w_0/2 - \frac{\eta_t^s}{s} C_{f,s}^{A}\\
  &\geq & e^{-\gamma}w_0/2 (1-1/s).  \EEAS Besides, $T_f$ being the
  number of full steps and $T$ the number of iterations before
  Fractional Away-step Frank-Wolfe stops,
\[
f(x_0) - f(x_T) \geq T_f e^{-\gamma}w_0/2 (1-1/s).
\]
Combining this with $f(x_0)-f(x_T) \leq f(x_0) - f(x^*) \leq w_0$ we get
\[
T_f \leq 2 e^{\gamma} \frac{s}{s-1}.
\]
Finally with the classical counting argument on drop steps, we obtain
\[
T \leq 4 e^{\gamma} \frac{s}{s-1} + |\mathcal{S}_0| - |\mathcal{S}_T|~.
\]
Denote $R$ the number of calls to the Fractional Away-step Frank-Wolfe before the last output $\hat{x}_R$ satisfies $e^{-\gamma}w(\hat{x},\mathcal{S}_{\hat{x}})/2 > C_{f,s}^{A}$. The strong Wolfe gap of the $N^{th}$ output of the Fractional Away-step Frank-Wolfe satisfies by definition
\[
w(\hat{x}_N) \leq e^{-N \gamma} w_0,
\]
hence we have
\[
R \leq \frac{1}{\gamma} \ln\Big(\frac{w_0}{2 C_{f,s}^{A}}\Big)~.
\]
Finally each round of the Fractional Away-step Frank-Wolfe under the initial assumption that $e^{-\gamma}w(\hat{x}_i,\mathcal{S}_{\hat{x}_i})/2 > C_{f,s}^{A}$ require at most $4 e^{\gamma} \frac{s}{s-1} + |\mathcal{S}_{\hat{x}_i}| - |\mathcal{S}_{\hat{x}_{i+1}}|$ iterations. Hence a total $T_t$ of
\BEAS
T_t &\leq & \sum_{i=1}^{R}{4 e^{\gamma} \frac{s}{s-1} + |\mathcal{S}_{\hat{x}_i}| - |\mathcal{S}_{\hat{x}_{i+1}}|}\\
&\leq & 4 R e^{\gamma} \frac{s}{s-1} + |\mathcal{S}_0|\\
&\leq & 4 \frac{s}{s-1} \frac{e^{\gamma}}{\gamma} \ln\Big(\frac{w_0}{2 C_{f,s}^{A}}\Big) + |\mathcal{S}_0|
\EEAS
which is the desired result.
\end{proof}

\section{Fractional Frank-Wolfe Algorithm}\label{sec:Frac_Frank_Wolfe}

In this section, we describe how H\"olderian error bounds coupled with a restart scheme yield improved convergence bounds for the vanilla Frank-Wolfe algorithm.

In Sections \ref{sec:sharpAWF}-\ref{sec:Holder}, relaxing strong convexity of $f$ using the $(\theta,c)$-\ref{eq:LGC} assumption lead to improved sub-linear rates using a restart scheme for the Away step variant of the Frank-Wolfe algorithm when the set of constraints $\mathcal{C}$ is a polytope. For these sets, away steps produce accelerated convergence rates that the vanilla Frank-Wolfe algorithm cannot achieve.

However, accelerated convergence hold for the vanilla Frank-Wolfe algorithm in other scenarios. For instance, when the solution of \eqref{eq:prob} is in the interior of the set and $f$ is strongly convex, the convergence of the vanilla Frank-Wolfe is linear.
In this vein, we define a fractional version of the Frank-Wolfe algorithm (Algorithm \ref{alg:frac-FW}) and analyse its restart scheme (Algorithm \ref{alg:restart-FW}) under the $(\theta,c)$-\ref{eq:LGC} condition in Section~\ref{ssec:optimum_interior}.
Although the Fractional Frank-Wolfe algorithm and the vanilla Frank-Wolfe algorithm perform the same iterations, the restart scheme produces a much simpler proof of improved convergence bounds. The fractional variant is also the structural basis for recent competitive versions of the Frank-Wolfe algorithm \citep{braun2017lazifying}.

Another acceleration scenario for the vanilla Frank-Wolfe algorithm is when the set of constraints $\mathcal{C}$ is strongly convex. Under some restrictive assumption on $f$, the classical analysis \citep[(5) in Theorem 6.1]{levitin1966constrained} exhibits a linear convergence rate. Recently \citep{garber2015faster} have shown a general $\mathcal{O}(1/T^2)$ sub-linear rate when $f$ and $\mathcal{C}$ are strongly convex. We will state new rates for the case where $f$ satisfies $(\theta,c)$-\ref{eq:LGC} and $\mathcal{C}$ is strongly convex, to provide a complete picture.

For completeness, we would like to mention that $\delta$-scaling for the away step Frank-Wolfe algorithm does not apply in the case where $\mathcal{C}$ is a strongly convex set. Lemma~\ref{lem:Scaling_polytope} does not hold anymore, and $PWidth$ can tend to zero in this case. 

\subsection{Restart schemes for Fractional Frank-Wolfe}
We now state the \textit{fractional} version of the (vanilla) Frank-Wolfe algorithm. 
The Fractional Frank-Wolfe algorithm~\ref{alg:frac-FW} is derived from Algorithm \ref{alg:frac-AFW} by replacing $w(x_0,\mathcal{S}_0)$ with $g(x_0)$, as in \eqref{eq:Frank_Wolfe_Dual_Gap} and dropping the away step update.

\begin{algorithm}
  \caption{Fractional Frank-Wolfe Algorithm}
  \label{alg:frac-FW}
  \begin{algorithmic}[1]
    \REQUIRE
      A smooth convex function $f$.
      Starting point \(x_{0} \in \mathcal{C}\).
      LP oracle \eqref{eq:lmo} and schedule parameter $\gamma > 0$.
    \STATE $t := 0$
    \WHILE{$g(x_t) > e^{-\gamma} g(x_0) $}
      \STATE
        \(v_t := \LP{\mathcal{C}}(\nabla f(x_{t})) \text{ and } d_t^{FW} := v_t - x_t\)
        \STATE $x_{t+1} := x_t + \eta_t d_t^{FW}$ with $\eta_t\in [0,1]$ via line-search
        \STATE $t := t+1$
    \ENDWHILE
    \ENSURE \(x_{t} \in \mathcal{C}\) such that $g(x_t) \leq e^{-\gamma} g(x_0)$
  \end{algorithmic}
\end{algorithm}

A constant restart scheme using Algorithm \ref{alg:frac-FW} for its inner iteration, recovers the Scaling Frank-Wolfe algorithm \citep[Algorithm 7: Parameter-free Lazy Conditional Gradient]{braun2017lazifying} up to a slight reformulation with the additional $\Phi_t$ parameter. The two algorithms have the same restart structure. However, the Scaling Frank-Wolfe algorithm additionally uses a weaker oracle (a so-called Weak Separation Oracle) than the Linear Optimization Oracle that we employ here. More precisely, the Scaling Frank-Wolfe algorithm does not necessarily require $v_t$ to be the exact nor an approximate solution of the Linear Minimization Problem, but rather to satisfy the condition $\langle -\nabla f(x_t), \, v_t - x_t \rangle > \Phi_t e^{-\gamma}$. As a consequence, $g(x_t)$ is not computed and $\Phi_t$ is only an upper bound on $g(x_t)$. This explains the difference in line \ref{line:difference} of Algorithm \ref{alg:restart-FW}.

\begin{algorithm}
  \caption{Restart Fractional Frank-Wolfe Algorithm}
  \label{alg:restart-FW}
  \begin{algorithmic}[1]
    \REQUIRE
      A smooth convex function $f$ with curvature~\(C_f\).
      Starting point \(x_{0} \in \mathcal{C}\). $\epsilon > 0$,
      LP oracle \eqref{eq:lmo} and schedule parameter $\gamma > 0$.
    \STATE $t := 0$ and $\Phi_0 := g(x_0)$
    \WHILE{$g(x_t) > \epsilon $}
      \STATE
        \(v_t := \LP{\mathcal{C}}(\nabla f(x_{t})) \text{ and } d_t^{FW} := v_t - x_t\)
       \IF{ $\langle -\nabla f(x_t), \, v_t - x_t \rangle > \Phi_t e^{-\gamma}$}\label{line:choice_frac_FW}
       \STATE $x_{t+1} := x_t + \eta_t d_t^{FW}$ with $\eta_t\in [0,1]$ via line-search
       \STATE $\Phi_{t+1} := \Phi_t$
       \ELSE
       \STATE $\Phi_{t+1} := g(x_t)$ (hence $\Phi_{t+1} < \Phi_t e^{-\gamma})$ \label{line:difference}
       \ENDIF
       \STATE $t := t+1$
    \ENDWHILE
  \end{algorithmic}
\end{algorithm}

\subsection{Optimum in the Interior of the Feasible Set}\label{ssec:optimum_interior}
We first recall that when the optimal solutions of \eqref{eq:optim_problem} are in the interior of $\mathcal{C}$, a version of the~\eqref{eq:Scaling} inequality is automatically satisfied. 
\eqref{eq:FW-scaling} replaces $w(x)$ by $g(x)$ and can be interpreted as a scaling inequality tailored to the (vanilla) Frank-Wolfe algorithm.
Note that the $\delta$ parameter depends on the relative distance of the optimal set $X^*$ to the boundary of $\mathcal{C}$.
This property has already been extensively used in, e.g., \citep{guelat1986some,garber2013linearly,garber2016linear}.

\begin{lemma}[FW $\delta$-scaling when optimum is in interior \citep{guelat1986some}]\label{lem:Scaling_interior}
Assume $\mathcal{C}$ is convex and $f$ convex differentiable.
Assume $X^*\subset\textbf{Int}(\mathcal{C})$  and choose $z>0$ such that $B(x^*,z) \subset\mathcal{C}$ for all $x^*\in X^*$. Then for all $x\in\mathcal{C}$ such that $d(x,X^*)\leq \frac{z}{2}$ we have
\BEQ\label{eq:FW-scaling}\tag{FW-Scaling}
g(x) \geq \frac{z}{2} \|\nabla f(x)\|,
\EEQ
where $g(x)$ is the Frank-Wolfe (dual) gap as defined in \eqref{eq:Frank_Wolfe_Dual_Gap}.
\end{lemma}

\begin{proof}
For $x\in B(x^*, \frac{z}{2})$, we have $x- \frac{z}{2}\frac{\nabla f(x)}{ \|\nabla f(x)\| }\in\mathcal{C}$.
Denote $v$ the Frank-Wolfe vertex, we have  $g(x)\triangleq \langle -\nabla f(x), \, v - x \rangle$. By optimality of $v$, we hence obtain
\[
g(x) \geq \langle -\nabla f(x), \, x-\frac{z}{2}\frac{ \nabla f(x) }{ \|\nabla f(x)\| } - x \rangle = \frac{z}{2} \|\nabla f(x)\|~,
\]
which is the desired result.
\end{proof}


We now bound the convergence rate of Algorithm \ref{alg:restart-FW} in the following proposition.

\begin{proposition}[Convergence Rate of Restart Fractional FW]
Let $f$ be a smooth convex function with curvature $C_f$ as defined in~\eqref{eq:curvature}, satisfying $(\theta,c)$-\ref{eq:LGC} on $\mathcal{C}$. Assume there exists $z>0$ such that $B(x^*,z)\subset\mathcal{C}$ for all $x^*\in X^*$. Let $\gamma > 0$ and assume  $x_0$ is such that $e^{-\gamma} g(x_0) \leq C_f$ and  $f(x_0)-f^*\leq \big(\frac{z}{2}\big)^{\frac{1}{\theta}}$ (burn-in phase). Then the output of Algorithm \ref{algo:scheduled_fractional} satisfies ($r= \frac{1}{1-\theta}$)
\begin{equation*}
  \left\{
    \begin{split}
    f(x_T) - f^* &\leq g_0\frac{1}{\Big( 1 + T C_{\gamma}^r \Big)^{\frac{1}{2-r}}}~~\text{ when } 1 \leq r < 2\\
    f(x_T) - f^* &\leq  g_0 \exp\left(-\frac{\gamma}{e^{2\gamma}}\frac{T}{8 C_f \mu }\right) ~~\text{ when } r=2~,
    \end{split}
  \right.
\end{equation*}
after $T$ steps, with $g_0=g(x_0)$. Also, with $\mu = (cz/2)^{1/(1-\theta)}$ we write
\[
C_{\gamma}^r \triangleq \frac{e^{\gamma (2-r)}-1}{ 2 e^{2\gamma} C_f \mu g(x_0)^{r-2}}.
\]
\end{proposition}
\begin{proof}
First note that for all $t$, we have $d(x_t,X^*)\leq \frac{z}{2}$. Indeed $f(x_t)-f^*\leq f(x_{t-1})-f^*\leq \big(\frac{z}{2}\big)^{\frac{1}{\theta}}$. Hence by $(\theta,c)$-\ref{eq:LGC} we have 
\[
\min_{x^*\in X^*} \|x_t - x^*\| \leq (f(x_t) - f^*)^\theta \leq \frac{z}{2}.
\] 
We can now apply lemma \ref{lem:Scaling_interior} to get for all $x_t$
\[
g(x_t) \geq \frac{z}{2} \|\nabla f(x_t)\|,
\]
and as in Lemma \ref{lem:strong_wolfe_primal_gap}, \ref{eq:FW-scaling} and $(\theta, c)$-\ref{eq:LGC} leads to a Wolfe primal gap (with $\mu = (cz/2)^{1/(1-\theta)} >0$)
\[
f(x) - f^* \leq \mu g(x)^r,
\]
where $r=1/(1-\theta)$. The proof then follows exactly that of Fractional Away Frank-Wolfe and its restart schemes (see Proposition \ref{prop:max_iter_fracFW} and Theorem \ref{th:convergence_restart_schemes_linear}), replacing $w(x)$ with $g(x)$. The only change comes from the upper bound on $T$, the number of iterations needed for Fractional Frank-Wolfe to stop. We recall the key steps to get this bound and update its value. At each iteration
\[
f(x_t) - f(x_{t+1}) \geq \max_{\eta\in[0,1]} \{ \eta e^{-\gamma} g(x_0) - \frac{\eta^2}{2} C_f\},
\]
such that because of assumption $e^{-\gamma}g(x_0) < C_f$, we have
\[
f(x_t) - f(x_{t+1}) \geq \frac{1}{2} \frac{g(x_0)^2}{e^{2\gamma} C_f}.
\]
Hence on one side
\[
f(x_0) - f(x_T) \geq \frac{T}{2} \frac{g(x_0)^2}{e^{2\gamma} C_f}.
\] 
And on the other side, using the $r$-Wolfe primal bound $f(x_0) - f(x_T)\leq \mu g(x_0)^r$ and finally
\[
T \leq 2\mu C_f e^{2\gamma} g(x_0)^{r-2}.
\] 
The restart scheme is then controlled exactly as in the proof of \ref{th:convergence_restart_schemes_linear}.
\end{proof}

Assuming that $e^{-\gamma} g(x_0) \leq C_f$ and $f(x_0)-f^*\leq \big(\frac{z}{2}\big)^{\frac{1}{\theta}}$ simplify the statements and it is automatically satisfied after a burn-in phase.
However it is fundamental to assume that there exists $z>0$ s.t. $B(x^*,z)\subset\mathcal{C}$ for all $x^*\in X^*$.
Indeed this ensures that the optimal set is in the interior of~$\mathcal{C}$.
Note also that a robustness result similar to that of Proposition \ref{prop:robustness_choice_gamma} holds here.

\subsection{Strongly Convex Constraint Set}\label{ssec:strongly_convex_set}
When~$\mathcal{C}$ is strongly convex, strong convexity of $f$ leads to a better convergence rate than the sub-linear $\mathcal{O}(1/T)$.
The original analysis of \citep[(5) in Theorem 6.1]{levitin1966constrained} assumes $\|\nabla f(x)\|\geq \epsilon >0$ (irrespective of the strong convexity of $f$) and hence $(\theta,c)$-\ref{eq:LGC} cannot be understood as a relaxation of the assumption.
This analysis provides a linear convergence rate when the unconstrained minimum of $f$ is strictly outside of $\mathcal{C}$. \S \ref{ssec:optimum_interior} shows linear convergence when $x^*$ is in the interior of $\mathcal{C}$. Hence the remaining case is when the unconstrained minimum of $f$ is in $\partial\mathcal{C}$, the boundary of $\mathcal{C}$ (an arguably rare instance).

Recently, the analysis of \citep{garber2015faster} closes this gap by providing a general convergence rate of $\mathcal{O}(1/T^2)$ under a (slightly) weaker assumption than strong convexity of $f$ \citep[see (2)]{garber2015faster}.
Although the asymptotic rate regime of \citep{garber2015faster} is significantly less appealing than the linear convergence rate in \citep{levitin1966constrained} and hence seemingly a marginal improvement in term of applicability, the situation is a bit more complicated: the bound of \citep{garber2015faster} benefits from much better conditioning and can easily dominate other bounds near $\partial\mathcal{C}$. 
In particular, the conditioning of \citep{levitin1966constrained} depends on the $\epsilon$ lower bounding the norm of the gradient on the constraint set, which can be arbitrarily small. 
The analysis of \citep{garber2015faster} adapts to $(\theta,c)$-\ref{eq:LGC}, as was detailed in \citep{YiAdapativeFW} and we recall this below for the sake of completeness.

\begin{theorem}\label{th:FW_C_strongly_convex}
Consider $\mathcal{C}$ an $\alpha$-strongly convex set and $f$ a convex $L$-smooth function \eqref{eq:function_smoothness}. Assume $(\theta,c)$-\ref{eq:LGC} for $f$. Then the iterate of (vanilla) Frank-Wolfe is such that $f(x_T) - f(x^*) = \mathcal{O}\big(1/T^{1/(1-\theta)}\big)$ for $\theta\in[0, 1/2]$.
\end{theorem}
\begin{proof}
From \citep[Lemma 1]{garber2015faster}, $L$-smoothness of $f$ combines with $\alpha$-strong convexity of $\mathcal{C}$ gives
\[
h_{t+1} \leq h_t \cdot \text{max }\Big\{\frac{1}{2}, 1 - \frac{\alpha \|\nabla f(x)\|}{8 L}\Big\}~.
\]
On the other hand with $(\theta, c)$-\ref{eq:LGC} and by convexity of $f$, \eqref{eq:identity_LGC} applies
\begin{eqnarray*}
\Big(f(x) - f(x^*)\Big)^{1-\theta} &\leq& c\cdot \underset{y\in X^*}{\text{min }} \frac{\langle\nabla f(x), \, x- x^*\rangle}{\|x - x^*\|}\\
 &\leq& c~\|\nabla f(x)\|~.
\end{eqnarray*}
Note that with $\theta=1/2$, this is the sufficient condition \citep[(2)]{garber2015faster} implied by strong convexity that leads to $\mathcal{O}(1/T^2)$ convergence rates. Hence combining both we recover this recursive inequality for $h_t = f(x_t) - f(x^*)$
\[
h_{t+1} \leq h_{t}\cdot\text{max }\Big\{ \frac{1}{2}, 1 - \frac{\alpha}{8Lc}h_t^{1-\theta}\Big\}~.
\]
When $\theta=0$ (convexity), this leads to the classical $\mathcal{O}(1/T)$ rate. When $\theta=1/2$ the above recursion leads to a $\mathcal{O}(1/T^2)$ rate as in \citep[proof of Theorem 2]{garber2015faster}. With Lemma \ref{lem:solve_recu} in Appendix \ref{appendix:non_standard_rec_rel}, for any non-negative constants $(k, C)$, such that $\frac{2 - 2^{\beta}}{2^\beta -1} \leq k$ and $ \text{max}\{h_0 k^{1/\beta}, \frac{2}{\big( (\beta - (1-\beta)(2^\beta - 1)) M\big)^{1/\beta}} \} \leq C$ (with $M\triangleq \frac{\alpha}{8L}$), we have
\[
h_t \leq \frac{C}{(t+1)^{1/(1-\theta)}}.
\]
and the desired result.
\end{proof}

Theorem \ref{th:FW_C_strongly_convex} interpolates between the general $\mathcal{O}\big(1/T\big)$ rate for smooth convex functions and the $\mathcal{O}\big(1/T^2\big)$ rate for smooth and strongly convex functions.


\section{Conclusion}
We derived a variant of the Away-step Frank-Wolfe algorithm and showed improved complexity bounds when the strong Wolfe gap satisfies a generalized strong convexity condition. The {\L}ojasiewicz factorization lemma shows that this condition actually holds generically for some value of the parameters, producing complexity bounds of the form $O(1/\epsilon^q)$ with $q\leq 1$, thus smoothly interpolating between the complexity of the classical FW algorithm with rate $O(1/\epsilon)$ and that of the Away-step Frank-Wolfe with rate $O(\log(1/\varepsilon))$. Our method is adaptive to the value of the generalized strong convexity parameters and robustly yields optimal performance. 

\section*{Acknowledgements}
\label{sec:acknowledgements}
This research was partially funded by Deutsche Forschungsgemeinschaft (DFG) through the DFG Cluster of Excellence MATH+ and the Research Campus Modal funded by the German Federal Ministry of Education and Research (fund numbers 05M14ZAM,05M20ZBM).
Research reported in this paper was partially supported by NSF CAREER Award CMMI-1452463. T.K. acknowledges funding from the CFM-ENS chaire {\em les mod\`eles et sciences des donn\'ees}. AA is at CNRS \& d\'epartement d'informatique, \'Ecole normale sup\'erieure, UMR CNRS 8548, 45 rue d'Ulm 75005 Paris, France,  INRIA and PSL Research University. The authors would like to acknowledge support from the {\em ML \& Optimisation} joint research initiative with the {\em fonds AXA pour la recherche} and Kamet Ventures, as well as a Google focused award.
This paper is a journal version of the conference version \citep{kerdreux2019restarting}.
The authors are also thankful to J\'er\^ome Bolte for insights about subanalycity of the strong Wolfe gap when $\mathcal{C}$ is a polytope.

{\small \bibliographystyle{abbrvnat}
\bibsep 1ex
\bibliography{MainPerso}}

\clearpage

\appendix

\section{One shot application of the Fractional Away-step Frank Wolfe}
\label{sec:oneShot}
\noindent Fractional Away-step Frank-Wolfe output a point $x_t\in\mathcal{C}$ s.t. $w(x_t) \leq e^{-\gamma}w_0$. 
Hence, running once Fractional Away-step Frank-Wolfe with a large value of $\gamma$ allows finding an approximate minimizer with the desired precision.
The following lemma proves a sublinear $\mathcal{O}(1/T)$ convergence rate which corresponds to the convergence rate of the Frank-Wolfe algorithm. Importantly the rate does not depend on $r$. Hence there is no hope of observing linear convergence for the strongly convex case.

\begin{lemma}
Let $f$ be a smooth convex function,  $\epsilon > 0$ be a target accuracy, and $x_0 \in \mathcal{C}$ be an initial point. Then for any $\gamma > \ln \frac{w(x_0)}{ \epsilon}$, Algorithm \ref{alg:frac-AFW} satisfies:
$$f(x_T) - f^* \leq \epsilon,$$
for $T \geq \frac{2C_f^{\mathcal{A}}}{\epsilon}$.
\end{lemma}
\begin{proof}
We can stop the algorithm as soon as the criterion $w(x_t) < \epsilon$ in step 2 is met or we observe an away step, whichever comes first. In former case we have $f(x_t) - f^* \leq w(t) < \epsilon$, in the latter it holds
$$ f(x_t) - f^* \leq -\nabla f(x_t)(d_t^{FW}) \leq \epsilon/2 < \epsilon.$$
Thus, when the algorithms stops, we have achieved the target accuracy and it suffices to bound the number of iterations required to achieve that accuracy. Moreover, while running, the algorithm only executes Frank-Wolfe and we drop the FW superscript in the directions; otherwise we would have stopped.

From the proof of Proposition \ref{prop:max_iter_fracFW}, we have each Frank-Wolfe step ensures progress of the form
$$f(x_t) - f(x_{t+1}) \geq
\begin{cases}
\frac{\langle -\nabla f(x_t), \, d_t \rangle^2}{2C_f^{\mathcal{A}}} & \text{ if } \langle -\nabla f(x_t), \, d_t \rangle \leq C_f^{\mathcal{A}} \\
\langle -\nabla f(x_t), \, d_t \rangle - C_f^{\mathcal{A}} / 2 & \text{otherwise}.
\end{cases}
$$
For convenience, let $h_t \triangleq f(x_t) - f^*$. By convexity we have $h_t \leq \langle -\nabla f(x_t), \, d_t \rangle$, so that the above becomes
$$f(x_t) - f(x_{t+1}) \geq
\begin{cases}
\frac{h_t^2}{2C_f^{\mathcal{A}}} & \text{ if } h_t \leq C_f^{\mathcal{A}} \\
h_t - C_f^{\mathcal{A}} / 2 & \text{ otherwise}.
\end{cases},
$$
and moreover observe that the second case can only happen in the very first step: $h_{1} \leq h_0 - (h_0 - C_f^{\mathcal{A}} / 2) = C_f^{\mathcal{A}} / 2 \leq 2 C_f^{\mathcal{A}} / t$ for $t=1$ providing the start of the following induction: we claim
$h_t \leq \frac{2C_f^{\mathcal{A}}}{t}$.

Suppose we have established the bound for $t$, then for $t+1$, we have
$$h_{t+1} \leq \left(1-\frac{h_t}{2 C_f^{\mathcal{A}}}\right) h_t \leq \frac{2C_f^{\mathcal{A}}}{t} - \frac{2C_f^{\mathcal{A}}}{t^2} \leq \frac{2C_f^{\mathcal{A}}}{t+1}.$$
The induction is complete and it follows that the algorithm requires $T \geq \frac{2C_f^{\mathcal{A}}}{\epsilon}$ to reach $\epsilon$-accuracy.
\end{proof}

\section{Non Standard Recurrence Relation}\label{appendix:non_standard_rec_rel}
This is a technical Lemma derived in \citep[proof of Theorem 1. ]{YiAdapativeFW} and we repeat it here for the sake of completeness. It is used in Section \ref{sec:Frac_Frank_Wolfe}.

\begin{lemma}[Recurrence and sub-linear rates]\label{lem:solve_recu}
Consider a sequence $(h_n)$ of non-negative numbers.  Assume there exists $M>0$ and $0 < \beta \leq 1$ s.t.
\BEQ\label{eq:generic_recurrence_relation}
h_{t+1} \leq h_t ~\text{max}\{1/2, 1 - M h_t^{\beta} \},
\EEQ
then $h_T = \mathcal{O}\big(1/T^{1/\beta} \big)$. More precisely for all $t\geq 0$,
\[
h_t \leq \frac{C}{(t+k)^{1/\beta}}
\]
with $(k,C)$ such that $\frac{2 - 2^{\beta}}{2^\beta -1} \leq k$ and $ \text{max}\{h_0 k^{1/\beta}, 2(C^\prime/M)^{1/\beta}\} \leq C$, with $C^\prime\geq \frac{1}{\beta-(1-\beta)(2^\beta -1)}$.
\end{lemma}
\begin{proof}
Let $(k,C)$ satisfying the condition in Lemma \ref{lem:solve_recu}. Let's show by induction that 
\[
h_t \leq \frac{C}{(t+k)^{1/\beta}}~.
\]
For $t=0$, it is true because we assumed $h_0 k^{1/\beta} \leq C$. Let $t \geq 1$. Consider the case where the maximum in the right hand side of \eqref{eq:generic_recurrence_relation} is obtained with $\frac{1}{2}$, then
\begin{eqnarray*}
h_{t+1} &\leq& \frac{C}{(t+k+1)^{1/\beta}} \Big(\frac{t+k+1}{t+k}\Big)^{1/\beta} \frac{1}{2}\\
&\leq& \frac{C}{(t+k+1)^{1/\beta}},
\end{eqnarray*}
because $k\geq \frac{2 - 2^{\beta}}{2^\beta - 1}$. Otherwise we have 
\[
h_{t+1} \leq h_t (1 - M h_t^{\beta}).
\]
If $h_t \leq \frac{C}{2(t+k)^{1/\beta}}$, conclusion holds as before. Otherwise assume $\frac{C}{2(t+k)^{1/\beta}} \leq h_t \leq \frac{C}{(t+k)^{1/\beta}}$ and \eqref{eq:generic_recurrence_relation} implies
\begin{eqnarray*}
h_{t+1} &\leq& \frac{C}{(t+k)^{1/\beta}} \Big( 1 - M \Big(\frac{C}{2}\Big)^{\beta} \frac{1}{t+k}\Big)\\
h_{t+1} &\leq& \frac{C}{(t+k+1)^{1/\beta}} \Big(1 + \frac{1}{t+k}\Big)^{1/\beta}\Big( 1 - M \Big(\frac{C}{2}\Big)^{\beta} \frac{1}{t+k}\Big)
\end{eqnarray*}
From Lemma \ref{lem:technical_rec}, for $C\geq 2(C^\prime/M)^{1/\beta}$ with $C^\prime\geq \frac{1}{\beta-(1-\beta)(2^\beta -1)}$, we have
\[
h_{t+1} \leq \frac{C}{(t+k+1)^{1/\beta}} \Big(1 + \frac{C^\prime}{t+k}\Big)\Big( 1 - \frac{C^\prime}{t+k}\Big) \leq \frac{C}{(t+k+1)^{1/\beta}}~,
\]
which proves the induction.
\end{proof}

\begin{lemma}\label{lem:technical_rec}
For any $t\geq 1$, we have
\BEQ\label{eq:techincal_rec}
\Big(1+ \frac{1}{t+k}\Big)^{1/\beta} \Big(1 - M \Big(\frac{C}{2}\Big)^\beta \frac{1}{t+k}\Big) \leq \Big(1 + \frac{C^\prime}{t+k}\Big)\Big(1 - \frac{C^\prime}{t+k}\Big),
\EEQ
where $k\geq \frac{2 - 2^{\beta}}{2^\beta - 1}$, $\beta\in]0,1]$, $C^\prime\geq \frac{1}{\beta-(1-\beta)(2^\beta -1)}$ and $C$ such that $C\geq 2 (C^\prime/M)^{1/\beta}$.
\end{lemma}

\begin{proof}
Write $x = \frac{1}{t+k}$. Because $t\geq 1$ and $k\geq \frac{2 - 2^{\beta}}{2^\beta - 1}$, we have $x\in ]0, 2^{\beta}-1]$. \eqref{eq:techincal_rec} is equivalent to
\[
\frac{1}{\beta}\log(1+ \frac{1}{t+k}) + \log(1 - M \Big(\frac{C}{2}\Big)^\beta x) \leq \log(1 + C^\prime x) + \log(1 - C^\prime x)
\]
Choosing $C$ greater or equal to $2 \Big(C^\prime/M\Big)^{1/\beta}$ ensures that $\log(1 - M \Big(\frac{C}{2}\Big)^\beta x)\leq \log(1 - C^\prime x)$. 
Also for $C^\prime\geq \frac{1}{\beta-(1-\beta)(2^\beta -1)}$, the function $h(x)\triangleq \log(1 + C^\prime x) - \frac{1}{\beta}\log(1 + x)$ is non-decreasing (and hence non-negative) on $]0,2^\beta -1]$.
Reciprocally for $C^\prime\geq \frac{1}{\beta-(1-\beta)(2^\beta -1)}$ and $C\geq \Big(C^\prime/M\Big)^{1/\beta}$, \eqref{eq:techincal_rec} holds.
\end{proof}

\end{document}